\documentclass[12pt]{amsart}
\usepackage{amsmath,amsfonts,amsthm,amscd,amssymb,mathrsfs,amssymb}
\usepackage{mathrsfs}
\usepackage{graphicx}
\usepackage[all]{xy}
\newtheorem{theorem}{Theorem}

\newtheorem{proposition}[theorem]{Proposition}
\newtheorem{lemma}[theorem]{Lemma}
\newtheorem{definition}[theorem]{Definition}

\newtheorem{corollary}[theorem]{Corollary}

\newtheorem{remark}[theorem]{Remark}
\newcommand{\CP}{\mathbb{CP}}
\newcommand{\CC}{\mathbb{C}}

\newcommand{\RR}{\mathbb{R}}
\newcommand{\ZZ}{\mathbb{Z}}
\newcommand{\RP}{\mathbb{RP}}
\newcommand{\U}{{\rm{U}}}
\newcommand{\LB}{{\rm{LB}}}
\renewcommand{\H}{\mathcal{H}}
\newcommand{\w}{\omega}
\newcommand{\W}{\Omega}
\renewcommand{\i}{i}
\newcommand{\ol}{\overline}
\newcommand{\lra}{\longrightarrow}
\newcommand{\lras}{\,\longrightarrow\,}

\newcommand{\set}{\,|\,}

\numberwithin{equation}{section}
\numberwithin{theorem}{section}
\begin{document}
\bibliographystyle{alpha} 
\title[Conformal symmetries]{Conformal symmetries of self-dual\\[2pt]
hyperbolic monopole metrics}
\author{Nobuhiro Honda}
\address{Department of Mathematics, Tokyo Institute of Technology, O-okayama,  Tokyo, Japan}
\email{honda@math.titech.ac.jp}
\author{Jeff Viaclovsky}
\address{Department of Mathematics, University of Wisconsin, Madison, 
WI, 53706}
\email{jeffv@math.wisc.edu}
\thanks{The first author has been partially supported by the Grant-in-Aid for Young Scientists  (B), The Ministry of Education, Culture, Sports, Science and Technology, Japan. The second author 
has been partially supported by the National Science Foundation under 
grant DMS-0804042.}
\begin{abstract}
We determine the group of conformal automorphisms 
of the self-dual metrics on $n \# \CP^2$ due to LeBrun 
for $n \geq 3$, and Poon for $n=2$.
These metrics arise from an ansatz involving a circle 
bundle over hyperbolic three-space $\H^3$ minus a finite number 
of points, called monopole points. We show that for $n \geq 3$ connected sums, any 
conformal automorphism is a lift of an isometry of $\H^3$ which
preserves the set of monopole points.
Furthermore, we prove that for $n = 2$, such lifts form a subgroup of 
index $2$ in the full automorphism group, which we show is a 
semi-direct product $( \U(1) \times  \U(1)) \ltimes {\rm{D}}_4$, 
the dihedral group of order $8$.
\end{abstract}
\date{February 11, 2009}
\maketitle
\setcounter{tocdepth}{1}
\vspace{-5mm}
\tableofcontents

\section{Introduction}
In \cite{Poon1986} and \cite{Poon1992}, Yat-Sun Poon found examples of 
self-dual conformal classes on the connected sums
$\CP^2 \# \CP^2$ and $3 \#\CP^2$
using techniques from algebraic geometry. In \cite{LeBrun1991}, Claude LeBrun gave a 
more explicit construction 
of $\U(1)$-invariant self-dual conformal classes on $n \# \CP^2$
for any $n$. Briefly, the idea is to choose $n$ distinct points 
$\{ p_1 , \dots, p_n \}$ in hyperbolic $3$-space $\H^3$, 
and consider a certain $\U(1)$-bundle $X_0 \rightarrow M_0$, 
where $M_0 = \H^3 \setminus \{ p_1, \dots, p_n \}$. 
A scalar-flat K\"ahler metric is written explicitly on $X_0$ in terms 
of a connection $1$-form, and extends to the metric completion
of $X$ of $X_0$, which is biholomorphic to $\CC^2$ blown 
up at $n$ points along a line. This metric conformally compactifies 
to give a self-dual conformal class on $\hat{X} = n \# \CP^2$,
which we denote by $[g_{\LB}]$. 
It turns out that any hyperbolic isometry which 
preserves the set of monopole points lifts to 
a conformal automorphism of $(n \# \CP^2, [g_{\LB}])$.
The main result of this paper is that the 
converse is also true for $n \geq 3$, and when $n=2$, such lifts form a 
subgroup of index $2$ in the full conformal group.
\begin{theorem}
\label{main1}
Let $n \geq 3$, and $[g_{\LB}]$ be any LeBrun self-dual conformal
class on $\hat{X} = n \# \CP^2$. A map $\Phi: \hat{X} \rightarrow \hat{X}$
is a conformal automorphism if and only if it is the
lift of an isometry of $\H^3$ which preserves the
set of monopole points.

For $n = 2$, there is a conformal involution 
$\Lambda : \hat{X} \rightarrow \hat{X}$
with the following property. For any conformal automorphism 
$\Phi: \hat{X} \rightarrow \hat{X}$, exactly one of $\Phi$ or 
$\Phi \circ \Lambda$ is the lift of an isometry of $\H^3$ which 
preserves the set of the two monopole points.
\end{theorem}
\begin{remark}{\em
The involution $\Lambda$ arises as follows. 
For $n=2$, there are exactly {\em{two}} semi-free conformal
$S^1$-actions,  which yield a double fibration of 
an open subset of $\hat{X}$ over 
$\H^3 \setminus  \{{\text{two points}}\}$.
The map $\Lambda$ interchanges 
the fibers of these two fibrations. We moreover 
find an $S^1$-family of involutions with the same 
properties, this will be proved in Section~\ref{geometric}.
To visualize this map, it is well-known that $\CP^2 \# \CP^2$ can 
be viewed as a boundary connect sum of two
Eguchi-Hanson ALE spaces (glued along the boundary $\RP^3$-s). 
The involution $\Lambda$ interchanges the Eguchi-Hanson spaces, and 
has an invariant $\RP^3$ (with fixed point set an $S^2$). 
The existence of such an automorphism is not difficult from the 
topological perspective, but finding one that is {\em conformal} 
is highly nontrivial. }
\end{remark}

We will let ${\rm{Aut}}(g)$ denote the conformal automorphism 
group, and ${\rm{Aut}}_0(g)$ denote the identity component. 
Theorem \ref{main1} implies the following. 
\begin{theorem}
\label{main2} Let $[g_{\LB}]$ be any LeBrun 
self-dual conformal class on $\hat{X} = n \# \CP^2$ and $n \geq 2$. 
All conformal automorphisms are orientation preserving.

If the monopole points do not lie on any common geodesic, 
then 
\begin{align}
\label{finitecase}
{\rm{Aut}}(g_{\LB}) = \U(1) \ltimes G, \ {\rm{Aut}}_0(g_{\LB}) = \U(1),
\end{align}
where $G$ is a finite subgroup of ${\rm{O}}(3)$.

If the monopole points all lie on a common hyperbolic 
geodesic, then 
\begin{align}
\label{u1u1}
{\rm{Aut}}_0 (g_{\LB})= \U(1) \times \U(1).
\end{align}
In this case, for $n \geq 3$ the full conformal group is
\begin{align}
\label{sd1}
{\rm{Aut}}(g_{\LB}) = {\rm{Aut}}_0(g_{\LB}) \ltimes  \ZZ_2,   
\end{align}
unless the point are configured symmetrically about a midpoint,
in which case 
\begin{align}
\label{sd2}
{\rm{Aut}}(g_{\LB})=  {\rm{Aut}}_0 (g_{\LB}) \ltimes ( \ZZ_2 \oplus \ZZ_2).
\end{align} 

In the case $n = 2$, \eqref{u1u1} necessarily holds, 
and the full conformal group is 
\begin{align}
\label{sd3}
{\rm{Aut}}(g_{\LB})= {\rm{Aut}}_0(g_{\LB}) \ltimes {\rm{D}}_4, 
\end{align} 
where ${\rm{D}}_4$ is the dihedral group of order $8$. 
\end{theorem}
The symmetry condition in the case of \eqref{sd2} is, 
more precisely, that there exists an extra hyperbolic 
reflection preserving the set of monopole which fixes only a 
midpoint on the common geodesic.
We will give explicit generators for each of the finite 
subgroups appearing in the semi-direct products \eqref{sd1}-\eqref{sd3}, 
see Theorem \ref{summarytheorem}.

 We next give a brief outline of the paper. 
We review the construction of LeBrun metrics in Section~\ref{Lansatz},
and we will detail the procedure for 
lifting hyperbolic isometries to conformal 
automorphisms of the LeBrun metrics. 
In Section~\ref{explicit}, we present an explicit 
form of the LeBrun metrics in the toric case, 
and discuss the extra involution in the 
case $n =2$. In Subsection~\ref{summary}, we give a
summary of the results, and give a short discussion 
of the fixed point set of involutions and invariant 
sets, and the action on cohomology.

The remainder of the paper will use twistor methods to prove 
that there are no other conformal 
automorphisms. Section \ref{twistor} will cover the case of $n \geq 3$, 
while Section~\ref{2CP2} will cover the case when $n =2$. 
The case of $n \geq 3$ is relatively easy, since in this case a (rational) 
quotient map for
the $\mathbb C^*$-action on the twistor space corresponding to the semi-free 
$\U(1)$-action 
is induced by a {\em{complete}} linear system, which implies
that any automorphism descends to the quotient space.
For $n = 2$, this is not true, and for this reason we instead use 
Poon's model of the twistor space, which is a small resolution of 
the intersection of two quadrics in $\CP^5$, see Section~\ref{2CP2}. 
In Subsection \ref{ss:autoproj}, 
we show that the holomorphic automorphisms of the intersection of the 
two quadrics which commute with the real structure consist of 16 tori. 
In Subsection \ref{ss:detsr}, we determine explicitly which small 
resolutions actually give the twistor space.
Then in Subsection \ref{detconf}, we show that the conformal automorphism 
group of Poon's metric consists of 8 tori, by explicitly determining which 
automorphisms among the 16 tori lift to the small resolutions obtained 
in Subsection \ref{ss:detsr}. Finally, we interpret these automorphisms 
geometrically in Section~\ref{geometric}, focusing on the 
involution $\Lambda$ when $n=2$.

 We could have alternatively started the paper with the sections 
on twistor theory-- this completely determines the automorphism group using 
only algebraic methods. However, 
one would like to understand the automorphisms geometrically, so 
we begin with the metric definition. From this perspective, it is 
easier to visualize the automorphisms for $n \geq 3$, as they are lifts 
of hyperbolic isometries. However, the existence of the extra conformal 
involution for $n =2$ is not at all obvious from the metric 
perspective (in fact we first discovered it from the twistor viewpoint). 

\subsection{Acknowledgements}
The authors are very grateful to Claude LeBrun for 
valuable discussions on his hyperbolic ansatz. 
The authors would also like to thank Simon Donaldson 
and Simon Salamon for insightful remarks. 
\section{Hyperbolic monopole metrics}
\label{Lansatz}
 We briefly recall the construction of LeBrun's self-dual 
hyperbolic monopole metrics from \cite{LeBrun1991}.
Consider the upper half-space model of hyperbolic 
space
\begin{align}
\H^3 = \{  (x,y,z) \in \RR^3,  z > 0 \},
\end{align}
with the hyperbolic metric $g_{\H^3} = z^{-2} ( dx^2 + dy^2 + dz^2)$. 
Choose $n$ distinct points  $p_1, \dots, p_n$ in 
$\H^3$, and let $P = p_1 \cup \dots \cup p_n$. 
Let $\Gamma_{p_j}$ denote the fundamental solution for the hyperbolic 
Laplacian based at $p_j$ with normalization 
$\Delta \Gamma_{p_j} = -2 \pi \delta_{p_j}$,
and let 
$V = 1 + \sum_{i = 1}^n  \Gamma_{p_i}$.
Then $* dV$ is a closed $2$-form on $\H^3 \setminus P$,
and $(1/ 2 \pi)[* dV]$ is an integral class in $
H^2 ( \H^3 \setminus P, \ZZ )$.
Let $\pi: X_0 \rightarrow \H^3 \setminus P$ 
be the unique principal $\U(1)$-bundle determined by the 
the above integral class.
By Chern-Weil theory, there is a connection form $\w \in H^1(X_0, \i \RR)$
with curvature form $\i (* dV)$. LeBrun's metric is defined by 
\begin{align}
\label{LBmetric}
g_{\LB} = z^2 (  V \cdot g_{\H^3} - V^{-1} \w \odot \w).
\end{align}
Note the minus sign appears, since by convention our connection 
form is imaginary valued.
We define a larger manifold $X$ by attaching points $\tilde{p_j}$ 
over each $p_j$, and by attaching an $\RR^2$ at $z = 0$. 
The space $X$ is non-compact, and has the topology 
of an ALE space. Adding the point at infinity will result in a 
compact manifold $\hat{X}$. 
\begin{remark}{\em Choosing a different connection form will 
result in the same metric, up to diffeomorphism, 
see the proof of Proposition \ref{gebe} below. }
\end{remark}
We summarize the main properties of $(X, g_{\LB})$ in 
the following proposition.  
\begin{proposition}[LeBrun \cite{LeBrun1991}]
The metric $g_{\LB}$ extends to $X$ as a smooth
Riemannian metric. 
The space $(X, g_{\LB})$ is asymptotically flat 
K\"ahler scalar-flat with a single end, and is 
biholomorphic to $\CC^2$ blown up  $n$ 
points on a line. By adding one point,
this metric conformally
compactifies to a self-dual 
conformal class on the compactification $(\hat{X}, [ g_{\LB}])$, 
which is diffeomorphic to $n \# \CP^2$.
\end{proposition}

We next review some facts from bundle theory, which will then 
be applied to LeBrun's metrics. 
\subsection{Bundle methods}
In this section $\U(1) \rightarrow X_0 \overset{\pi}{\rightarrow} M$ will be 
a principal $\U(1)$-bundle over a connected oriented base manifold $M$. 
The group $\U(1)$ acts on $X_0$ from the right, we will denote 
this action by $R_{g}$ for $g \in \U(1)$. 
Recall that a connection $\w \in \Lambda^1(X_0; \i \RR)$ is a $1$-form 
on $X_0$ with values 
in the Lie algebra of $\U(1)$. The connection satisfies 
(i) $\w$ restricted to the fiber $\pi^{-1}(z)$ is 
$\i \cdot d \theta$, the 
Maurer-Cartan form on $\U(1)$, and (ii) $R_g^* \w = \w$. 
Since the group is abelian,  
the {\em{curvature $2$-form}} of the connection is given 
by $\W_{\w} = d \w \in H^2(X, \i \RR)$, and this forms descends to $M$. 
\begin{definition}{\em
The connections $\w$ and $\w'$ are said to be {\em{gauge equivalent}} 
if there exists a function $f : M \rightarrow \RR$ such 
that $\w = \w' +  \i \cdot df$. }
\end{definition}
\begin{remark}{\em
\label{h1}
If $\W_{\w} = \W_{\w'}$ then 
$d( \w - \w') = 0$. If $H^1(M; \RR) = 0$, then 
$\w - \w' = \i \cdot df$, so $\w$ and $\w'$ are gauge equivalent. }
\end{remark}
\begin{definition} \label{def:g-equiv}
{\em
The connections $\w$ and $\w'$ are said to be {\em{bundle equivalent}} if there exists a 
fiber-preserving map 
$B: X_0 \rightarrow X_0$ covering the identity map of $M$, 
that is $\pi \circ B = \pi$, and which commutes with the 
right action of $\U(1)$, satisfying $B^* \w' = \w$. } 
\end{definition}
\begin{proposition}
\label{gebe}
If the connections $\w$ and $\w'$ are gauge equivalent then 
they are bundle equivalent. The converse holds if $H^1(M, \RR) = 0$.  
\end{proposition}
\begin{proof}
If the connections are gauge equivalent, then $ \w =  \w' +  \i \cdot d f$.
Define a bundle map $B : X_0 \rightarrow X_0$ by
$B v =  v \cdot e^{ i f}$ (right action).  Letting
$\w_1'$ denote a local connection form on the base, we have
\begin{align}\label{g_to_b}
B^* \w' = B^* (  \w_1' + \i \cdot d \theta) 
=  \w_1' + \i B^* d \theta 
=  \w_1'+ \i( d \theta +  d f )
= \w' +  \i  \cdot d f = \w. 
\end{align}
Conversely, if $B^* \w' = \w$, then 
$\W_{\w} = d \w = d B^* \w' = B^* \W_{\w'}$. 
These are forms on the base, and $B$ covers the 
identity map, so 
$\W_{\w} = \W_{\w'}$, which implies that $\w$ and $\w'$
are gauge equivalent by Remark \ref{h1}. 
\end{proof}
Since $X_0$ is a $\U(1)$-bundle, it has a first Chern 
class $c_1(X_0) \in H^2(M ; \ZZ)$. From the exponential 
sheaf sequence, $H^1(M, \mathcal{E}^*) = H^2(M ; \ZZ)$, 
so $X_0$ is determined up to smooth bundle equivalence by $c_1(X_0)$. 
By Chern-Weil theory, the image of $c_1(X_0)$ in $H^2(M; \i \RR)$ 
is cohomologous to $\W_{\w}$, for any connection $\w$ on 
$X_0$.
\begin{proposition}
\label{liftprop}   
Assume that $H^1(M;\ZZ) = 0$, and that $H^2(M; \ZZ)$ has no torsion.  
Let $\w$ be a connection on $X_0$,
and $\phi: M \rightarrow M$ an orientation preserving diffeomorphism satisfying
$\phi^* \W_{\w} = \W_{\w}$.  
Then there exists 
an equivariant lift of $\phi$ to $\Phi: X_0 \rightarrow 
X_0$ satisfying $\Phi^* \omega = \omega$. 
If $\phi: M \rightarrow M$ is an orientation reversing diffeomorphism satisfying
$\phi^* \W_{\w} = -\W_{\w}$,  then there exists 
such a lift satisfying $\Phi^* \omega = - \omega$. 
These lifts are unique up to right action by a 
constant in $\U(1)$. In both cases, $\Phi$ is orientation 
preserving. 
\end{proposition}
\begin{proof}
First assume that $\phi$ is orientation preserving. 
Consider the pull-back bundle  $\phi^* X_0$.
By naturality, 
\begin{align}  c_1( \phi^* X_0) = \phi^* c_1(X_0) = \phi^* [ \W_{\w} ] = [\W_{\w}] = c_1(X_0).
\end{align} 
Consequently, there exists a bundle equivalence $A : \phi^* X_0 \rightarrow X_0$,
which is an equivariant map covering the identity map on $M$. 
Denote by $\pi_2$ the natural map $\pi_2 : \phi^* X_0 \rightarrow X_0$. 
This is summarized in the following diagram. 
\begin{equation}
\label{cd1}
 \CD
X_0@>{A^{-1}}>>\phi^*X_0@>{\pi_2}>>X_0\\
@VVV @VVV @VV{\pi}V\\
M@>{Id}>>M@>{\phi}>>M.\\
 \endCD
 \end{equation}
The pull-back $\w' = (A^{-1})^* \pi_2^* \w$ is a connection on $X_0$. 
Since $\pi_2 \circ A^{-1}$ covers $\phi$, we have
\begin{align}
\W_{\w'} = d \w' = d \big( ( \pi_2 \circ A^{-1})^* \w \big)
=  ( \pi_2 \circ A^{-1})^* \W_{\w} = \phi^* \W_{\w} = \W_{\w}.
\end{align}
From Remark \ref{h1}, it follows that $\w'$ and $\w$ are gauge equivalent.
By Proposition \ref{gebe}, $\w'$ and $\w$ are bundle 
equivalent, so there exists a bundle map $B: X_0 \rightarrow X_0$ 
satisfying $B^* \w' = \w$. 
The desired map is given $\pi_2 \circ A^{-1} \circ B$.
In the construction of the map $B$ in the proof of 
Proposition \ref{gebe} above, there is a freedom 
to replace the function $f$ by $f + c$ for any constant $c$, 
and the uniqueness statement follows. 

 If $\phi$ is orientation reversing, then the pull-back bundle
$\phi^* X_0$ will satisfy $c_1( \phi^* X_0) = - c_1 (X_0)$.
In this case we need to add an additional map identifying 
the bundle with its conjugate bundle using complex 
conjugation, which corresponds geometrically to making a reflection 
in each fiber (such a choice is not canonical). Clearly, this makes the lift orientation 
preserving. 
\end{proof}
\begin{remark}
\label{explrm}
{\em
These lifts can be computed explicitly once 
the transition functions of the bundle are known 
(with respect to some open cover). 
Assume that the bundle is trivialized over a simply 
connected open set $U$, and that $U$ is a $\phi$-invariant set. 
Tracing through the above proof, to find the lift, we 
must first find a function $f : U \rightarrow \RR$ such that 
\begin{align}
\phi^* \omega - \omega = \i \cdot df, 
\end{align}
and the lift is then right multiplication 
by $e^{ i f}$ in each fiber (if $\phi$ is orientation-reversing, 
then we add a reflection in each fiber). The action in other 
coordinate systems is then found using the transition functions. }
\end{remark}
\begin{proposition}
\label{fixedfiber}
Let $p$ be a fixed point of $\phi$. 
If $\phi$ is orientation reversing, then any lift $\Phi$ of $\phi$
fixes exactly $2$ points over $p$.
\end{proposition}
\begin{proof}
From the above proof, any lift is a reflection in the fiber over
a fixed point. A reflection always has exactly $2$
fixed points. 
\end{proof}
\subsection{Lifts of hyperbolic isometries}
We only give a brief summary in this section;
for background on hyperbolic geometry, see \cite{Ratcliffe}. 
Using the quaternions, write hyperbolic upper half space as
\begin{align}
\H^3 = \{ x + y i + z j 
\set  (x,y,z) \in \RR^3,  z > 0  \}.
\end{align}
The group of hyperbolic isometries is the group
of time-oriented Lorentz transformations ${\rm{SO}}_{+}(3,1)$. 
The identity component is isomorphic to 
${\rm{PSL}}(2, \CC)$; an isomorphism 
can be seen explicitly as follows \cite[Chapter 4]{Ratcliffe}. 
Any orientation preserving hyperbolic isometry $\phi: \H^3 \rightarrow \H^3$
may be written 
\begin{align}
\label{ori} 
\phi(w) = (a w + b)( c w + d)^{-1}, 
\end{align}
with $(a,b,c,d) \in \CC^4$, and $ad - bc =1$, where we use quaternionic multiplication 
and inverse.  Similarly, any orientation reversing hyperbolic 
isometry $\phi: \H^3 \rightarrow \H^3$
may be written
\begin{align} 
\label{rori}
\phi(w) = (a (- \bar{w} ) + b)( c (-\bar{w} ) + d)^{-1},
\end{align}
with $(a,b,c,d) \in \CC^4$, and $ad - bc =1$. 
\begin{proposition} 
\label{hyppoints}
Let $\{p_1, \dots, p_n \} \subset \H^3$, and $n \geq 2$. 
Let $G$ denote  the group of all hyperbolic isometries preserving 
this set of points.
If all points lie on a single hyperbolic 
geodesic $\gamma$, then $G = {\rm{O}}(2)$ acting as rotations 
and reflections about $\gamma$,
unless the points are configured symmetrically about a midpoint, 
in which case $G = {\rm{O}}(2) \oplus {\rm{O}}(2)$
(more precisely, this symmetry condition is that there is 
another reflection preserving the set of points, and 
$G$ is generated by ${\rm{O}}(2)$ and this reflection). 
Finally, if the points do not lie on any common geodesic, 
then $G$ is conjugate to finite subgroup of ${\rm{O}}(3)$.
\end{proposition}
\begin{proof}
This can be proved by a direct computation using the 
presentations (\ref{ori}) and (\ref{rori}).
The proof is finished by noting that {\em{any}} finite subgroup of 
${\rm{SO}}_{+}(3,1)$ is conjugate to a subgroup of ${\rm{O}}(3)$,
see \cite[Theorem 5.5.2]{Ratcliffe}.
\end{proof}
The following proposition shows the lifts obtained
in Proposition \ref{liftprop} yield conformal automorphisms 
of LeBrun's metrics.
\begin{proposition}
\label{hypconf}
If $\phi: \H^3 \rightarrow \H^3$ is a hyperbolic isometry
preserving the set of monopole points, then there exists a
unique $\U(1)$-family of lifts $\Phi$ as in Proposition \ref{liftprop} 
which are orientation preserving conformal automorphisms of 
$(X_0, g_{\LB})$. Furthermore, any such lift extends to a conformal 
automorphism of the compactification $(n \# \CP^2, [g_{\LB}])$. 
\end{proposition} 
\begin{proof} 
If $\phi$ is a hyperbolic isometry, then
\begin{align}
V \circ \Phi = 1 + \sum_{j=1}^n \Gamma_{\Phi^{-1}(p_j)}.
\end{align}
Since $\phi$ fixes the set of monopole points, 
we have $\phi^* V = V$. 
We chose the connection above so that 
$\W_{\w} = \i (* d V)$. This implies that 
$\phi^* \W_{\omega} = \W_{\omega}$ if $\phi$ is 
orientation preserving, and $\phi^* \W_{\omega} = - \W_{\omega}$ 
if $\phi$ is orientation reversing. 
In either case, we may apply Proposition 
\ref{liftprop} to find a lift 
of $\phi$ satisfying $\Phi^* \omega =\pm \omega$.
By assumption, $\phi^* g_{\H^3} = g_{\H^3}$, so we have
\begin{align}
\begin{split}
\Phi^* g_{\LB} &= (z \circ \Phi)^2 \Big( (V \circ \Phi) \cdot \Phi^* g_{\H^3} 
- (V \circ \Phi)^{-1} \cdot \Phi^*(\omega \odot \omega) \Big)\\
& =  (z \circ \Phi)^2 \Big(  V \cdot g_{\H^3} 
- V^{-1} \omega \odot \omega \Big) =  
\Big(\frac{z \circ \Phi}{z} \Big)^2 g_{\LB}.
\end{split}
\end{align}
For the last statement, the $S^1$-action of fiber rotation
on $X_0$ clearly extends smoothly to the compactification,
since $\hat{X}$ is obtained from $X_0$ by adding 
fixed points over the monopole points, and also adding
the entire boundary of $\H^3$, which is also
fixed by the $S^1$-action. The argument in \cite{LeBrun1991} 
for extending the metric conformally to $\hat{X}$ generalizes to 
show that $\Phi$ yields a {\em{smooth}} conformal 
diffeomorphism of $\hat{X}$, we omit the details. 
\end{proof}
We emphasize that Proposition \ref{liftprop} only provides a lift
of a single isometry.  The lifting of a group of isometries is 
more subtle. We define ${\rm{Aut}}(\H^3;p_1,\dots,p_n) \ $ to be the group of 
isometries of $\H^3$ which preserve the set of monopole points,
and let ${\rm{Aut}}(g_{\LB}; p_1, \dots, p_n) \ $ denote the subgroup of 
conformal automorphisms which are lifts of elements in 
${\rm{Aut}}(\H^3;p_1,\dots,p_n)$. Clearly, we have an exact sequence
\begin{align}
\label{exact}
1 \rightarrow \U(1) \rightarrow {\rm{Aut}}(g_{\LB}; p_1, \dots, p_n )
\overset{\rho}{\rightarrow} {\rm{Aut}}(\H^3;p_1,\dots,p_n),
\end{align}
where $\rho$ is the obvious projection. 
A natural question is whether this sequence {\em{splits}}, 
that is, does there exist a homomorphism 
\begin{align}
\mu:  {\rm{Aut}}(\H^3;p_1,\dots,p_n) \rightarrow  {\rm{Aut}}(g_{\LB}; p_1, \dots, p_n)
\end{align}
such that $\rho \circ \mu = \mbox{id}$? 

In general, this sequence does not split. We next give a 
condition for the sequence to split when restricted to a subgroup of 
$G \subset {\rm{Aut}}(\H^3;p_1,\dots,p_n)$. 
\begin{proposition}
\label{u1lifts}Let the subgroup $G$ consist of 
orientation preserving elements. If the subgroup $G$ has a fixed 
point $p \in \H^3 
\setminus \{p_1, \dots, p_n\}$, then there is a 
splitting homomorphism 
\begin{align}
\label{gamsplit}
\mu:  G \rightarrow  {\rm{Aut}}(g_{\LB}; p_1, \dots, p_n), 
\  \rho \circ \mu = {\rm{id}}_G.
\end{align}
Furthermore,
\begin{align}
\U(1) \times G \subset {\rm{Aut}}(g_{\LB}; p_1, \dots, p_n).
\end{align}
\end{proposition}
\begin{proof}
Since $p$ is not one of the monopole points, then 
any element of $G$ has a unique lift which fixes
the fiber over $p$. This defines the
splitting map $\mu$. 
To see that $\mu$ is a homomorphism: given 
$g_1 \in G$ and $g_2 \in G$, we compare
$\mu(g_1 g_2)$ with $\mu(g_1) \mu(g_2)$. 
The former is, by definition, the lift 
of $g_1 g_2$ in the unique lift which fixes the fiber over 
$p$. The latter is also a lift of $g_1 g_2$, and 
fixes the fiber over $p$, since both $\mu(g_1)$
and $\mu(g_2)$ fix this fiber. By uniqueness, 
they are the same. 

Next, $\U(1)$ is the identity component, which is normal.
We claim that $\mu(G)$ is also normal. 
To see this, let $\mu(g) \in \mu(G)$, and $\Phi \in  
{\rm{Aut}}(g_{\LB}; p_1, \dots, p_n)$. Then $\Phi \mu(g) \Phi^{-1}$
fixes $\{p_1, \dots, p_n \}$ and fixes the fiber over $p$, therefore must be the
of the form $\mu(h)$ for some element $h \in G$. 

Finally, since both subgroups are normal, by an elementary theorem 
in group theory, we therefore have a direct product.
\end{proof}
\begin{remark}
\label{finitermk}
{\em
Consider the case when the points are not 
contained on a common geodesic. 
Then ${\rm{Aut}}(\H^3;p_1,\dots,p_n)$ is conjugate 
to a subgroup of ${\rm{O}}(3)$. 
Let us assume for simplicity that 
the symmetry group $G$ is conjugate to a subgroup 
of ${\rm{SO}}(3)$. The group $G$ either 
fixes a geodesic, or has a single fixed 
point. In the former case, there must be a non-monopole 
fixed point, and Proposition  \ref{u1lifts} can be 
applied. In the latter case, if the fixed point is 
not a monopole point, then again Proposition \ref{u1lifts} 
can be applied. But if the fixed point is a monopole
point, then the entire group might
not lift. In this case, it is possible that the group 
$G$ appearing in \eqref{finitecase} is a strictly smaller 
subgroup of  ${\rm{Aut}}(\H^3;p_1,\dots,p_n)$, and which 
might not necessarily lift to a normal subgroup. We do not know 
of any such example for which this happens, however.}
\end{remark}
\begin{proposition}
\label{commprop}
If all of the monopole points lie on a common geodesic,
then 
\begin{align}
\U(1) \times {\rm{SO}}(2) = \U(1) \times  \U(1)  \subseteq {\rm{Aut}}(g_{\LB}).
\end{align}
\end{proposition}
\begin{proof}
The subgroup ${\rm{SO}}(2)$ of rotations around a geodesic 
fix the entire geodesic. Let $p$ be any non-monopole
point on the geodesic, and apply Proposition \ref{u1lifts}. 
\end{proof}
In the next section we present a direct method of finding such lifts, 
via an explicit connection form.
\section{An explicit global connection}
\label{explicit}
In this section we give an explicit connection for the 
LeBrun ansatz in the toric case. 
We first consider the case of $2$ monopole points. 
Let the monopole points lie on the $z$-axis, $p_1 = (0,0,c_1)$, and $p_2 = (0,0,c_2)$,
with $c_1 < c_2$. Choose cylindrical coordinates 
\begin{align}
\H^3 = \{ (x,y,z) = ( r \cos \theta_3, r \sin \theta_3, z ) , z > 0\}.
\end{align}
\begin{theorem}
\label{conn}
Let $U = \H^3 \setminus \{z\mbox{-axis}\}$,
and write  
\begin{align}
\H^3 \setminus \{p_1, p_2 \} = U_1 \cup U_2 \cup U_3, 
\end{align}
where 
\begin{align}
U_1 &= U \cup \{ (0,0,z), z < c_1 \} = U \cup I_1,\\
U_2 &= U \cup \{ (0,0,z), c_1 < z < c_2 \} = U \cup I_2,\\
U_3 &= U \cup \{ (0,0,z), z > c_2 \} = U \cup I_3.
\end{align}
Let $f_c : \H^3 \setminus \{p_1, p_2\} \rightarrow \RR$ 
denote the function
\begin{align}
f_c(r,z) = \frac{c^2 - r^2 - z^2}{2 \sqrt{ (c^2 + r^2 + z^2)^2 - 4 c^2 z^2}}.
\end{align}
 Then $f = f_{c_1} + f_{c_2}$ satisfies
\begin{align}
d (f d \theta_3 ) = * d V, 
\end{align}
in $U$.  That is, the form $ i f d \theta_3$
is a local connection form in $U$.
Define
\begin{align}
\omega_1(x) &= \i (-1 + f) d \theta_3, x \in U_1,\\
\omega_2(x) &= \i f d\theta_3,  x \in U_2, \\ 
\omega_3(x) &= \i (1 + f) d \theta_3, x \in U_3. 
\end{align}
These $1$-forms define a global connection (with values
in $\mathfrak{u}(1) = \i \RR$) on 
the total space $X_0 \rightarrow M$. 
That is, there is a global connection $\omega$ on $X_0$,
such that over $U_j$, $\omega$ has the form 
$\omega_j +   \i\cdot d \theta_1$,
where $\theta_1$ is an angular coordinate on the fiber.
\end{theorem}
\begin{proof}
Recall we want the connection to have curvature form $\Omega_{\omega} = * d V$,
where $V = 1 + \Gamma_{p_1} + \Gamma_{p_2}$.
The Green's function is given by
\begin{align}
\Gamma_{(0,0,c)} (x,y,z) 
& =  - \frac{1}{2} + \frac{1}{2}
\left[
1 - \frac{ 4 c^2 z^2}{\left( r^2 + z^2 + c^2 \right)^2}
\right]^{-1/2},
\end{align}
where $r^2 = x^2 + y^2$, see \cite[Section 2]{LeBrun1991b}.
An important point is that $\Gamma$ only depends upon $z$ and $r$. 
A computation shows that in cylindrical coordinates
\begin{align}
* d V = \left( \frac{r}{z}\left( - V_r dz + V_z dr  \right) \right) \wedge d \theta_3,
\end{align}
since $d * d V=0$, this implies that 
\begin{align}
d \left( \frac{r}{z}\left( -V_r dz + V_z dr \right) \right) = 0. 
\end{align}
The first quadrant $Q_1 = \{ (r, z), r > 0, z > 0 \}$ is contractible, 
so there exists a function $f = f(r,z)$ such that
\begin{align}
d f =  \frac{r}{z}\left( -V_r dz + V_z dr \right).
\end{align}
We let 
\begin{align}
\kappa =  \frac{r}{z} V_z dr -  \frac{r}{z}   V_r dz = \kappa_1 dr + \kappa_2 dz.
\end{align}
An explicit potential $f$ satisfying $df = \kappa$ is  
\begin{align}
f = \left( \int_0^1 \kappa_1( tr, tz) dt \right)r
+ \left( \int_0^1 \kappa_2( tr, tz) dt \right)z.
\end{align}
A computation, which we omit, shows that 
\begin{align}
f = f_{c_1} + f_{c_2},
\end{align}
is a solution where $f_c$ is given by
\begin{align}
f_c(r,z) = \frac{c^2 - r^2 - z^2}{2 \sqrt{ (c^2 + r^2 + z^2)^2 - 4 c^2 z^2}},
\end{align} 
and any other solution differs from 
this by a constant, since $U$ is connected.
An important remark is that 
\begin{align}
  (c^2 + r^2 + z^2)^2 - 4 c^2 z^2 \geq 0, 
\end{align}
and if $ (c^2 + r^2 + z^2)^2 - 4 c^2 z^2 = 0$, then 
$(z-c)^2 + r^2 = 0$, so $f_c$ is well-defined 
on all of $\H^3 \setminus \{(0,0,c)\}$.  
We then have on $U$,
\begin{align}
d ( f d \theta_3) = d f \wedge d \theta_3 = * d V. 
\end{align}
Along the $z$-axis, we have the following
\begin{align}
f(0,z) =
\begin{cases}
1  &   z < c_1,\\
0  &  c_1 < z < c_2,\\
-1 &   z > c_2. \\
\end{cases}
\end{align}
Furthermore, $\frac{\partial}{\partial r} f (0,z) = 0$. 
Consequently, $(f-1) d \theta_3$ is a smooth $1$-form in $U_1$,
$\omega_2 = f d \theta_3$ is a smooth $1$-form in $U_2$,
and $\omega_3 = (1 + f) d \theta_3$ is a smooth $1$-form in $U_3$. 
\end{proof}
\begin{remark}{\em
For $n > 2$, simply take $f = f_{c_1} + \dots + f_{c_n}$,
and $U_j$ to be the union of $U$ with the corresponding 
interval on the $z$-axis, with the connection form in 
each chart adding the appropriate constant multiple of $d \theta_3$. 
We remark that an explicit potential in the case $n=2$ was
written down in \cite{GibbonsWarnick} in pseudospherical 
coordinates, but only in a single chart; our method
above yields a {\em global} connection form. }
\end{remark}
We can use the above to write down explicit transition 
functions for the bundle $X_0 \rightarrow M$. 
\begin{proposition}
\label{trans}
With respect to the covering $\{ U_1, U_2, U_3 \}$, the transition 
functions of the bundle are given by $g_{21} = e^{ \i \theta_3}$, 
and $g_{23} = e^{- \i \theta_3}$. 
\end{proposition}
\begin{proof}
From above
\begin{align}
\omega_2 - \omega_1 = \i f d \theta_3 - \i (-1 + f)  d\theta_3 
= \i\cdot d \theta_3. 
\end{align}
The formula for the change of connection is given by 
\begin{align}
\omega_2 - \omega_1 = g_{21}^{-1} d g_{21},
\end{align}
which implies that $g_{21} = e^{\i \theta_3}$. 
Also,
\begin{align}
\omega_2 - \omega_3 = \i (f d \theta_3) - \i (1 + f) d \theta_3 
= - \i\cdot d \theta_3 = g_{23}^{-1} d g_{23},
\end{align}
which implies that $g_{23} =  e^{- i\theta_3}$.
\end{proof}
We name two points on the boundary of
$\H^3$: $q_1 = (0,0,0)$, and $q_2 = (0,0,\infty)$. 
We denote the union of the fibers over
$\overline{I_j}$ by $\Sigma_j$ $(1\le j\le 3)$, which is 
a $2$-sphere. We also let $\Sigma_4$ denote
the $2$-sphere corresponding to the boundary 
of hyperbolic space. 
Using the above, we next show that the $S^1$-action on $\H^3$ 
given by rotation around the $z$-axis
has infinitely many lifts to conformal $S^1$-actions on $(\hat{X}, [g_{\LB}])$. 
We recall that a {\em{semi-free}} action is a non-trivial action 
of a group $G$ on a connected space $M$ such that for every $x \in M$, the 
corresponding isotropy subgroup is either all of $G$ or is trivial. 
\begin{proposition}
Consider the $S^1$-action on $\H^3$
by oriented rotations around the $z$-axis.
Then for any integer $k$, there exists a lift to a
conformal $S^1$-action on $(\hat{X}, [g_{\LB}])$ 
such that $e^{\i \theta}$ lifts with following 
property: the lifted action rotates the fibers over $I_2$ by $e^{\i k \theta}$,
it rotates the fibers over $I_1$ by $e^{\i (k-1) \theta}$,
and rotates the fibers over $I_3$ by $e^{\i (k+1) \theta}$.

Consequently, for $k = 0$, the lift fixes only 
$\Sigma_2 \cup \{q_1,q_2\}$, and this action is the
only only lift to a semi-free $S^1$-action. 
For $k = 1$, the fixed points are 
$\Sigma_1 \cup \{p_2, q_2\}$,
and for $k = -1$, the lift fixes only $\Sigma_3 \cup \{p_1, q_1\}$.
For any other $k$, the fixed set consists of  
four points $\{p_1, p_2, q_1, q_2\}$.
\end{proposition}
\begin{proof}
Let $\phi$ denote an oriented rotation about the $z$-axis,
determined by $e^{\i \theta_0}$.  
As in the above section, we know a lift of $\phi$,
call it $\Phi$, exists, and is unique up to a right multiplication by a 
constant. If we choose the lift $\Phi$ so that $\Phi$ fixes 
fiber over a point on $I_2$, then $\Phi$ fixes 
all fibers over $I_2$. This follows 
because the connection form on $U_2$ chosen in 
Theorem \ref{conn} is invariant under rotations 
around the $z$-axis, see Remark \ref{explrm}. 
From the transition functions given in 
Proposition \ref{trans}, $\Phi$ rotates the
fibers over $I_1$ by $e^{-\i \theta_0}$, and the 
fibers over $I_3$ are rotated by $e^{+\i \theta_0}$.
Finally, it is clear that we can lift to an $S^1$-action by specifying 
the action on the fibers over $I_2$; there is a 
lift for any integer $k$ so that the fibers of 
$I_2$ are rotated by $e^{\i k \theta}$. 
The semi-free claim is obvious, since for $k=0$, the lift 
only makes a single rotation on $\Sigma_1$ and $\Sigma_3$,
while for $k \neq 0$, $I_1$ and $I_3$ are rotated 
multiple times.  
Again, the argument in \cite{LeBrun1991} extends to show all of the 
above actions yield smooth actions on the compactification $\hat{X}$, 
we omit the details. 
\end{proof}
We denote the lifted action for $k =0$ by $K_3$. 
Since the $K_3$-action clearly commutes with the 
$K_1$-action, this then gives an identification of the identity 
component of the automorphism group with $K_1 \times K_3$, 
where $K_1$ is the group of rotations in the fiber. 
It will be shown below in Lemma \ref{lemma:sf}, that for $n=2$, 
$K_1$ and $K_3$ are the only semi-free $S^1$-actions. 
We will also see in Section \ref{geometric} that the 
$K_3$-action yields another fibration of an open subset of $X$ 
over $\H^3 \setminus \{ \text{two points} \}$. 

While for simplicity of presentation we restricted the above discussion 
to the case of $2$ monopole points,  it is clear that for the case of 
$n$ monopole points all lying on a common geodesic, 
the $SO(2)$-action of rotations in $\H^3$ around the geodesic 
will have a lift to an $S^1$-action for any integer 
$k$. Since these actions commute with the fiber rotation,
there is a torus action as identity component. 
However, in contrast to $n =2$, for $n \geq 3$, none of these 
lifted $S^1$-actions are semi-free, see Lemma \ref{lemma:sf}. 
\subsection{Relation to Joyce metrics}
In \cite{Joyce1995},  Dominic Joyce gave a 
very general construction for toric self-dual metrics. 
In that paper, on page 547, Joyce remarks that toric LeBrun metrics
are in fact Joyce metrics. However, a direct argument is not 
provided. In this short section (which is really just 
an extended remark) we suggest a direct argument, 
using the explicit connection form from the previous section. 
For space considerations, we do not provide full details.
Again, for simplicity we consider only that case $n=2$.  
With this choice of connection, the LeBrun metric
takes the form
\begin{align}
g_{\LB} &= V \frac{1}{z^2} ( dr^2 + r^2 d \theta_3^2 + dz^2) - V^{-1} \omega_2^2,
\end{align}
which we can write as
\begin{align}
V^{-1} g_{\LB} &= \frac{1}{z^2} ( dr^2 + dz^2) + \frac{r^2}{z^2} d \theta_3^2
+ V^{-2} ( f d \theta_3 + d \theta_1)^2\\
& = g_{\H^2} + \frac{r^2}{z^2} d \theta_3^2
+ V^{-2} ( f d \theta_3 + d \theta_1)^2.
\end{align}

Recall that the coordinates $(r,z)$ lie in the first quadrant $Q_1$. 
Let $(x_1, x_2)$ be coordinates in upper half space $\H^2$. 
Under the conformal change of coordinates
\begin{align}
x_1 &= r^2 - z^2,\\
x_2 &= 2 r z,
\end{align}
we have
\begin{align}
dx_1^2 + dx_2^2 = 4(r^2 + z^2) ( dr^2 + dz^2) = 4 \sqrt{ x_1^2 + x_2^2} ( dr^2 + dz^2).
\end{align}
Solving for $r$ and $z$,
\begin{align}
r^2 &= \frac{1}{2} \left( x_1 + \sqrt{ x_1^2 + x_2^2} \right),\\
z^2 &=  \frac{1}{2} \left( -x_1 + \sqrt{ x_1^2 + x_2^2} \right).
\end{align}
Consequently, in these coordinates, $g_{\LB}$ is conformal to 
\begin{align}
g_{\H^2} + \frac{2(x_1^2 + x_2^2)}{x_2^2}
\left[  \left( 1  + \frac{x_1}{ \sqrt{x_1^2 + x_2^2}} \right) d \theta_3^2
+  \left( 1  - \frac{x_1}{ \sqrt{x_1^2 + x_2^2}} \right)  
 V^{-2} ( f d \theta_3 + d \theta_1)^2 \right].
\end{align}
This looks very close to Joyce's ansatz for toric self-dual metrics
over hyperbolic $2$-space $\H^2$ (with $x_1$ and $x_2$ switched). 
However, the weights of the torus action here are not exactly the 
same as chosen in Joyce's paper.
Identifying the torus action using $K_1 \times K_3$,  
the stabilizer subgroups are given as follows:  
\begin{align}
\Sigma_1: (1, -1),  \Sigma_2: (1,0),  \Sigma_3: (1,1), \Sigma_4: (0,1).
\end{align}
However,  Joyce's weights for the boundary spheres are
given by (notation as in \cite{Joyce1995})
\begin{align}
(m_1, n_1) = (0,1), (m_2, n_2) = (1,2), (m_3, n_3) = (1,1), (m_4, n_4) = (1,0), 
\end{align}
The change of angular coordinates $\theta_1 \mapsto \theta_3$, 
$\theta_3 \mapsto \theta_1 + \theta_3$ (which corresponds to an 
outer automorphism of ${\rm{SL}}(2, \ZZ)$), then brings the
boundary weights into agreement. After this change, it is 
then possible to show directly this is indeed a Joyce metric, but we omit the 
computation. A similar argument also works for $n \geq 3$ monopole points. 

\subsection{Extra involution for $\mathbf{n=2}$}
\label{extra}
Recall we have the boundary sphere $\Sigma_4$ 
fixed by $K_1$, and the sphere $\Sigma_2$, fixed by $K_3$.
We next find a conformal transformation which interchanges these
spheres, and also has the property that $p_1$ maps to $q_1 = (0,0,0)$, and 
$p_2$ maps to $q_2 = (0,0,\infty)$. 
This map will interchange the orbits of the $K_1$ and $K_3$ actions. 

Let $r,z$ and $c_1,c_2$ have the same meaning as in the beginning of Section \ref{explicit}.
We first define an automorphism $\varphi : Q_1 \rightarrow Q_1$,
where $Q_1 = \{(r,z)\set r > 0, z > 0 \}$ is the first 
quadrant. 
\begin{definition}{\em 
Let $w = r + \i z$, and define}
\begin{align}
\label{map}
\varphi(w) = \i c_2 \sqrt{ \frac{ \bar{w}^2 + c_1^2}{ \bar{w}^2 + c_2^2}}
= \varphi_1(r,z) + \i \varphi_2(r,z).
\end{align}
\end{definition}
\begin{proposition}
The map $\varphi$
interchange $I_2$ and $I_4$, interchanges $p_1$ and $q_1 = (0,0,0)$, 
and interchanges $p_2$ and $q_2 = (0,0,\infty)$ (thinking of the above 
as contained in $\overline{Q}_1$).  
Under the identification of $Q_1$ with the upper half plane 
under the complex square $w \mapsto w^2$, the map is a hyperbolic 
isometry. 
\end{proposition}
\begin{proof}
We identify $Q_1$ with $\H^2$ using the complex square,
\begin{align}
\zeta = x_1 + \i x_2 = ( r + \i z)^2= s(w). 
\end{align}
Under this map, the monopole points $p_j$ map to 
$(- c_j^2, 0)$. 
Consider the M\"obius transformation defined by
\begin{align}
\label{mobius}
L(\zeta) = - (c_2)^2 \frac{ \bar{ \zeta} + (c_1)^2}{ \bar{ \zeta} + (c_2)^2},
\end{align}
which is an orientation-reversing hyperbolic 
isometry of $\H^2$. It has the property that 
\begin{align}
\label{Lpoints}
L((  - c_1^2, 0)) = (0,0), \ L(0,0) = ( - c_1^2, 0), \mbox{ and } 
L(- c_2^2, 0)) = (\infty, 0). 
\end{align}
Clearly, $\varphi(w) = s^{-1} \circ  L \circ  s (w)$, which is \eqref{map}. 
The first statement follows easily. 
\end{proof}
\begin{remark}
\label{Lunique}
{\em
The map $L$ is the unique orientation-reversing 
hyperbolic involution satisfying \eqref{Lpoints}. }
\end{remark}
The original coordinates on $U_2 \times S^1$ are ordered 
$(r, \theta_3, z, \theta_1)$, but in the following we will 
rearrange coordinates so that this domain is $Q_1 \times S^1 \times S^1$. 
\begin{definition}{\em
For any angle $\vartheta$, define the map 
$\tilde{\Lambda}(\vartheta): X \rightarrow X$ by}
\begin{align}
\tilde{\Lambda}(\vartheta): \big( (r,z), \theta_3, \theta_1 \big) \mapsto 
\big( \varphi(r,z), \theta_1 - \vartheta, \theta_3 + \vartheta \big).
\end{align}
\end{definition}
On first observation, it might appear that the map 
$\tilde{\Lambda}(\vartheta)$ is not well-defined at points on 
the $z$-axis corresponding the the intervals $I_1$ and 
$I_3$, where the coordinate $\theta_3$ is 
not defined. However, the map is in fact well-defined
everywhere: 
\begin{proposition}
For any angle $\vartheta$, the map $\tilde{\Lambda}(\vartheta)$ extends to a diffeomorphic involution 
of $\hat{X} = 2 \# \CP^2$.
The extension interchanges $\Sigma_2$ and $\Sigma_4$, and 
interchanges the points $p_j$ and $q_j$ for $j = 1,2$.
\end{proposition}
\begin{proof}We need only consider the case that 
$\vartheta = 0$, since $\tilde{\Lambda}(\vartheta) = (e^{-i \vartheta}, e^{i \vartheta}) 
\cdot \tilde{\Lambda}(0)$ (viewing this as the $K_1 \times K_3$-action). 
We note that initially $\tilde{\Lambda}(0)$ is defined with respect 
to a trivialization of the bundle on the open set $U_2$. 
To confirm that it well-defined everywhere, we must use the transition 
functions from Proposition \ref{trans}. 
For example, in $U_2$, the angles change by 
$(\theta_3, \theta_1) \mapsto (\theta_1, \theta_3)$.
Taking into account the transition 
function $g_{21} = e^{\i \theta_3}$, in $U_1$
the action is $(\theta_3, \theta_1) \mapsto (\theta_1 - \theta_3, \theta_1)$.
In the $U_1$ chart, the map $\tilde{\Lambda}(0)$ therefore takes the form 
\begin{align}
( r, z, \theta_3, \theta_1) \mapsto ( \varphi_1 (r,z), \varphi_2 (r, z), \theta_1 - \theta_3,
\theta_1). 
\end{align}
Rewriting the map in the coordinates $(x,y,z,\theta_1)$, 
\begin{align}
\label{phit}
(x,y,z,\theta_1) \mapsto (  \varphi_1 (r,z) \sin( \theta_1 - \theta_3), 
 \varphi_1 (r,z) \cos( \theta_1 - \theta_3), 
 \varphi_2 (r,z), \theta_1).
\end{align} 
For points with $r=0$, the map $\varphi$ is given by 
\begin{align}
\label{phiz}
\varphi(0,z) = \left( 0, c_2 \sqrt{  \frac{ c_1^2 - z^2}{c_2^2 - z^2}} \ \right),
\end{align}
which {\em is} well-defined on $I_1$. 
Therefore, for $(x,y) = (0,0)$, \eqref{phit} becomes
\begin{align}
(0,0, z,  \theta_1) \mapsto (0,0, \varphi(0,z), \theta_1 ),
\end{align}
which is indeed well-defined. A similar argument confirms that 
$\tilde{\Lambda}(0)$ is well-defined (and smooth) everywhere 
on $2 \# \CP^2$. 

It is easy to see that $\tilde{\Lambda}(0)$ interchanges 
$\Sigma_2$ and $\Sigma_4$, and interchanges the points $p_j$ and $q_j$ for 
$j = 1,2$. Finally, it is clear that $\tilde{\Lambda}(\vartheta)$ is an 
involution.
\end{proof}
\begin{theorem}
For any angle $\vartheta$, the map $\tilde{\Lambda}(\vartheta)$ is a conformal 
involution of $[g_{\LB}]$. 
\end{theorem}
\begin{proof}
It would be a formidable calculation to show directly 
that this map is indeed conformal. 
In this paper, for space considerations, we therefore prefer to 
argue indirectly using twistor theory, see Theorem \ref{final} below. 
\end{proof}

\subsection{Summary}
\label{summary}

In this section, we summarize what we have obtained 
so far, and we also make some remarks about the fixed point sets of 
various lifts. 

\begin{theorem}
\label{summarytheorem}
Consider $(n \# \CP^2, [g_{\LB}])$ and $n \geq 2$. 
If the monopole points do not lie on any common geodesic (so that $n\ge 3$), 
then 
\begin{align}
\label{i0}
\U(1) \ltimes G \subseteq {\rm{Aut}}(g_{\LB}), 
\end{align}
where $G$ is a finite subgroup of ${\rm{O}}(3)$.

Next, assume that the monopole 
points all lie on a common geodesic. Let ${\rm{Aut}}_0$ 
denote the identity component of ${\rm{Aut}}(g_{\LB})$. Then we have
\begin{align}
\label{i1}
 \U(1) \times  \U(1) = {\rm{Aut}}_0(g_{\LB}).
\end{align}

Let $\phi_3$ be any reflection 
about a hemisphere on which all the monopole points belong.
Then there exists a lift $\Phi_3$ of $\phi_3$ which 
is also an involution. Let $\ZZ_2 = \{ Id, \Phi_3 \}$ denote the subgroup 
generated by $\Phi_3$. Then the semi-direct product 
\begin{align}
\label{i2}
 (\U(1) \times  \U(1)) \ltimes \ZZ_2 \subseteq {\rm{Aut}}(g_{\LB}).
\end{align}

In the case there is an additional reflection symmetry $\phi_2$ 
(which is always the case for $n=2$), consider also 
the composition $\phi_1 = \phi_2 \circ \phi_3$. 
Then, in addition to $\Phi_3$, there exist lifts 
$\Phi_j$ of $\phi_j$, for $j = 1,2$ such that 
$\{ Id, \Phi_1, \Phi_2, \Phi_3 \}$ is a subgroup 
of ${\rm{Aut}}$ which is isomorphic to $\ZZ_2 \oplus \ZZ_2$,
and 
\begin{align}
\label{i3}
 (\U(1) \times  \U(1)) \ltimes ( \ZZ_2 \oplus \ZZ_2)  \subseteq {\rm{Aut}}(g_{\LB}).
\end{align}

For $n = 2$ consider also the extra involution $\tilde{\Lambda}(0)$.
Then 
\begin{align*}
\{ Id,  \Phi_1, \Phi_2, \Phi_3, \tilde{\Lambda}(0),  
\tilde{\Lambda}(0) \Phi_1, \tilde{\Lambda}(0) \Phi_2, 
\tilde{\Lambda}(0) \Phi_3 \}
\end{align*}
is a subgroup of ${\rm{Aut}}$ isomorphic to ${\rm{D}}_4$, the dihedral group with 
$8$ elements,  and 
\begin{align}
\label{i4}
 (\U(1) \times  \U(1)) \ltimes {\rm{D}}_4 \subseteq {\rm{Aut}}(g_{\LB}).
\end{align}
\end{theorem}
\begin{proof}
The inclusion \eqref{i0} was discussed above in Remark \ref{finitermk}. 
The equality \eqref{i1} follows from Proposition \ref{commprop},
and the fact that the identity component is a manifold, 
and cannot be strictly greater than dimension $2$ in 
this case \cite{Poon1994}. 

For \eqref{i2}, we let $\tilde{\Phi}_3$ be {\em{any}} lifting 
of $\phi_3$ from Proposition \ref{hypconf}. Note that 
$\tilde{\Phi}_3^2$ is orientation preserving and covers the 
identity map of $\H^3$. Therefore, by the uniqueness in 
Proposition \ref{hypconf}, we must have that $\tilde{\Phi}_3^2 = R(g)$
is right multiplication by $g \in \U(1)$. 
To find an involution, we then define $\Phi_3$ to be 
$\tilde{\Phi}_3 \circ R(\sqrt{g^{-1}})$. This is an 
involution since any lift is equivariant. Therefore $\{Id, \Phi_3\}$ is 
indeed a subgroup of ${\rm{Aut}}(g_{\LB})$ isomorphic to $\ZZ_2$. 
Since the identity component is necessarily normal, the group 
generated by the identity component and this $\ZZ_2$-subgroup 
is a semi-direct product. 

For \eqref{i3}, we let $\Phi_3$ be as in the previous paragraph. 
Next, the map $\phi_1 = \phi_2 \circ \phi_3$ is an orientation 
preserving hyperbolic 
isometry which fixes a geodesic. Thus we may apply 
Proposition \ref{gamsplit}, and let $\Phi_1 = \mu (\phi_1)$. 
Since $\phi_1$ is an involution, from the definition of $\mu$, 
it follows that $\Phi_1$ is also an involution. 
Then we {\em{define}} $\Phi_2 = \Phi_1 \circ \Phi_3$,
which is necessarily a lift of $\phi_2$.  
Clearly, $\{ Id, \Phi_1, \Phi_2, \Phi_3 \}$ is a 
subgroup isomorphic to  $\ZZ_2 \oplus \ZZ_2$,
and for the same reason as in the previous paragraph, 
the generated subgroup is the semi-direct product. 

Finally, the inclusion \eqref{i4} will be proved in Section 
\ref{geometric}, see Proposition \ref{D4prop}. 
\end{proof}
It is the purpose of Sections \ref{twistor} and \ref{2CP2} 
below to show that the inclusions \eqref{i2}-\eqref{i4} 
are in fact equalities. We end this section with a short discussion 
on fixed point 
sets of involutions, and the action on cohomology.
\begin{theorem} 
\label{fixedpoint}
For $(n \# \CP^2, [g_{\LB}])$ and $n \geq 2$, assume that the monopole 
points all lie on a common geodesic. 
If $\phi_3$ is a reflection 
about a hemisphere containing all the monopole points, 
then the lift $\Phi_3$ of $\phi_3$ given in Theorem \ref{summarytheorem}
has fixed point locus $\Upsilon_3 = n \# \RP^2$, which is contained in an 
invariant $ n \# \RP^3$. Furthermore,  $\Phi_3$ induces the 
identity map on cohomology. 

In the case there is an additional reflection symmetry $\phi_2$ 
(which is always the case for $n=2$), consider also 
the composition $\phi_1 = \phi_2 \circ \phi_3$. 
Let $\Upsilon_j$ denote the fixed locus of $\Phi_j$,
where $\Phi_j$ are the lifts of $\phi_j$ given in 
Theorem \ref{summarytheorem}.
For $n$ even, $\Upsilon_1$ and $\Upsilon_2$ are both two-dimensional 
spheres, and $\Upsilon_1 \cap \Upsilon_2 = S^1$. 
For $n$ odd,  $\Upsilon_1=  S^2$ and 
$\Upsilon_2 = \RP^2$, with $\Upsilon_1 \cap \Upsilon_2 = S^1$. 
The maps $\phi_1$ and $\phi_2$ induce the following map
on cohomology
\begin{align}
\big( k_1, k_2, \dots, k_{n} \big) \mapsto 
\big( k_n, k_{n-1}, \dots, k_1 \big),
\end{align}
with respect to an orthonormal basis of $H^2(n\#\mathbb{CP}^2;\mathbb Z)$.
Further, for $n$ even, $\Phi_2$, leaves invariant 
an $S^3$. For $n$ odd,  $\Phi_2$ leaves invariant an $\RP^3$. 

For $n = 2$, the fixed point set of the extra involution 
$\tilde{\Lambda}(0)$ is an $S^2$,  which is contained in an 
invariant $\RP^3$. Also, $\tilde{\Lambda}(0)$ induces the following 
map on cohomology:
\begin{align}
(k_1, k_2) \mapsto (-k_2, k_1),
\end{align}
with respect to an orthonormal basis of $H^2(2\#\mathbb{CP}^2;\mathbb Z)$.
\end{theorem}
\begin{proof}
We let $\phi_3$ be a reflection in a hemisphere containing 
the monopole points. Since $\phi_3$ is orientation reversing, 
by Proposition \ref{fixedfiber}, the lift $\Phi_3$ will fix exactly 
2 points in each fiber over this hemisphere. Let $\Upsilon_3$ denote the 
fixed locus.  Topologically, $\Upsilon_3$ is a double covering 
of a $2$-disc branched over the boundary circle and over $n$ points. 
We compute
\begin{align}
\chi ( \Upsilon_3) = 2 \chi(D^2) - \chi(S^1) - n = 2 - n.
\end{align}
It turns out that $\Upsilon_3$ is non-orientable, so by the surface classification, 
$\Upsilon_3 = n \# \RP^2$ (to see non-orientability, we note that 
odd dimensions is clear since the Euler characteristic is odd, 
and the even-dimensional case case be viewed as a limiting case of the next 
higher odd dimension). The invariant set is a circle 
bundle over this hemisphere, branched over 2 points
and the boundary circle, so is $n \# \RP^3$. 

When the points are in symmetric configuration, 
we let $\phi_2$ denote the extra symmetry of inversion 
in a hemisphere. 
If $n$ is even, there is no monopole point on this hemisphere. 
Since $\phi_2$ is orientation reversing, Proposition 
\ref{fixedfiber} implies that the fixed point set of the lift 
$\Phi_2$ is a double cover of $D^2$ branched only over the boundary circle,
so $\Upsilon_2 = S^2$. The invariant set is a circle 
bundle over the disc branched over the boundary, 
so is an $S^3$. 
Next, define $\phi_1 = \phi_2 \circ \phi_3$. 
The fixed point set of $\phi_1$ is a geodesic $\gamma$.
From the proof of Theorem \ref{summarytheorem},
our choice of the lifting $\Phi_1$ fixes a fiber over 
a point of $\gamma$, thus fixes every fiber over $\gamma$. 
Therefore, $\Upsilon_1$ is a circle 
bundle over $\gamma$, completed by adding two points
on the boundary of $\H^3$, so $\Upsilon_1 = S^2$. 
The intersection of $\Upsilon_1$ and $\Upsilon_2$ gives 
$2$ points in each fiber over $\gamma$. Adding the 2
boundary points gives that $\Upsilon_1 \cap \Upsilon_2 = S^1$. 

If $n$ is odd, then there is a monopole point on this
hemisphere. From Proposition \ref{fixedfiber}, the fixed point set 
of the lift $\Phi_2$ is a double cover of $D^2$ branched over the 
boundary circle, and a single point.  We have
\begin{align}
\chi( \Upsilon_2) = 2 \chi(D^2) - \chi(S^1) - 1 = 1, 
\end{align}
which implies that $\Upsilon_2$ is $\RP^2$. The invariant 
set is a circle bundle over $D^2$ branched over the 
boundary circle, and a single point, thus is an $\RP^3$. 
 
As in the even case, define 
$\phi_1 = \phi_2 \circ \phi_3$. Again,  the fixed point set of $\phi_1$ is a geodesic
$\gamma$. Therefore, $\Upsilon_1$ is contained in the restriction 
of the bundle to this geodesic (including the 2 boundary points 
of the geodesic). Since there is a single monopole 
point on this geodesic, the restriction of the bundle is topologically 
the wedge $S^2 \vee S^2$. From the proof of Theorem \ref{summarytheorem}, 
the lift $\Phi_1$ 
was chosen to fix a fiber over some point on this geodesic.  
Since the fixed point set must be a smooth $2$-dimensional manifold, 
$\Upsilon_1$ must be one of these $S^2$-s, depending upon the particular 
choice of the lift $\Phi_1$. 
The intersection of $\Upsilon_1$ and $\Upsilon_2$ then is 2 points 
in each fiber over one half of $\gamma$, together with the monopole 
point and a single boundary point, which implies that
$\Upsilon_1 \cap \Upsilon_2 = S^1$. 

In the case $n=2$, recall the hyperbolic isometry $L$ 
defined in \eqref{mobius}. It is easy to verify that the fixed 
point set of $L$ is given by
\begin{align}
(x_1 - c_2^2)^2 + x_2^2 = c_2^2 ( c_2^2 - c_1^2),
\end{align}
and is therefore a semicircle centered at $(-c_2^2, 0)$
of radius $c_2 \sqrt{ c_2^2 - c_1^2}$. 
Since $z \mapsto z^2$ is a conformal transformation, the
fixed point set of $\varphi$ is a semi-circle centered on the 
$z$-axis at $(0, c_2)$, intersecting the positive $z$-axis at two points,
one of them on the interval $I_1$, and the other on $I_3$.  
The fixed point set of $(\theta_3, \theta_1) \mapsto (\theta_1, \theta_3)$ 
is obviously points of the form $(\theta_3, \theta_3)$. Thus the 
fixed point set of $\tilde{\Lambda}(0)$ is a circle 
bundle over the semicircle branched over the two 
endpoints, therefore is an $S^2$. The invariant 
set consists of all the torus fibers over the 
semicircle, which is easily seen to be an $\RP^3$
(it is the $S^1$ bundle restricted to a sphere containing
both monopole points). 

These involutions can be visualized as follows. 
In the case $n = 2$, it is well-known that $\CP^2 \# \CP^2$ can 
also be viewed as a boundary connect sum of $2$ 
Eguchi-Hanson ALE space (glued along the boundary $\RP^3$-s). The involution $\Phi_1$ 
reverses the two factors of the usual connect sum,
and has an invariant $S^3$ (it flips $\Sigma_1$ and $\Sigma_3$), 
while the involution $\tilde{\Lambda}(0)$ 
interchanges the Eguchi-Hanson spaces, and has an 
invariant $\RP^3$ (it flips $\Sigma_2$ and $\Sigma_4$).  
For $n$ even, then involution $\Phi_1$ reflects the 
connect sum through the central neck of the connect 
sum, and has an invariant $S^3$. For $n$ odd,  
then involution $\Phi_1$ reflects the 
connect sum through a central $\CP^2$ summand, 
and has an invariant $\RP^3$.
The action on cohomology follows easily from these descriptions. 
\end{proof}

\begin{remark}{\em 
In the case of a single monopole point, the LeBrun conformal 
class compactifies to the conformal class of the Fubini-Study 
metric on $\CP^2$, which is Einstein.  By Obata's 
Theorem, any conformal automorphism is an isometry, thus
the conformal automorphism group for $n=1$ is ${\rm{SU}}(3)$. 
For $n = 0$, the LeBrun conformal 
class compactifies to the conformal class of the 
round metric on $S^4$, thus the conformal 
group is ${\rm{SO}}_{+}(5,1)$, the time-oriented 
Lorentz transformations. 
For $n \geq 1$, there are no orientation reversing 
diffeomorphisms, this follows from the Hirzebruch 
Signature Theorem since the signature is non-zero. 
However, $S^4$ does admit orientation-reversing 
diffeomorphisms, which is reflected in the 
fact that ${\rm{SO}}_{+}(5,1)$ has $2$ components. } 
\end{remark}

\section{LeBrun's twistor spaces}
\label{twistor}
Let ${\rm{Aut}}(\H^3;p_1,\dots,p_n)$ be the group of isometries of 
$\H^3$ which preserve the set of monopole points $p_1,\dots,p_n$.
In this section, we prove Theorem \ref{main1} in the 
case $n \geq 3$ by showing the following. 
\begin{proposition}\label{prop-hom}
Let $[g_{\LB}]$ be a LeBrun self-dual conformal class on $n\#\CP^2$ with 
monopole points $p_1, \dots,p_n\in \H^3$. Suppose $n\ge 3$.
Then there is a homomorphism 
\begin{align}
\rho:{\rm{Aut}}( g_{\LB})\lras
{\rm{Aut}}(\H^3;p_1,\dots,p_n)
\end{align}
such that $\rho (\Phi) = \phi$, where $\Phi$ is any 
lift of $\phi$ obtained in Proposition \ref{hypconf}. 
\end{proposition}
\begin{remark}{\em
In the previous sections, we used the upper half space model 
of hyperbolic space. However, in this and the following sections, $\H^3$ will 
no longer refer to any specific model of hyperbolic $3$-space.}
\end{remark}
In the following we prove Proposition \ref{prop-hom} by using twistor spaces;
for background on twistor theory, see \cite{AHS}, \cite{Besse}. 
Let $Z$ be the twistor space of $[g_{\LB}]$ in Proposition \ref{prop-hom}, 
and ${\rm{Aut}}(Z)$ the group of holomorphic transformations of $Z$.
By the twistor correspondence, there is a canonical injective homomorphism
\begin{align}\label{hom1}
{\rm{Aut}}(g_{\LB})\lras {\rm{Aut}}(Z)
\end{align}
(see, for example, \cite[Proposition 2.1]{PedersenPoon}). 
Using this, we regard ${\rm{Aut}}(g_{\LB})$ as a subgroup of ${\rm{Aut}}(Z)$.
Let $F$ be the canonical square root of $-K_Z$ (the anticanonical line bundle).
Then the action of ${\rm{Aut}}(g_{\LB})$ on $Z$ naturally lifts to
the line bundle $F$.
Hence we obtain a homomorphism 
\begin{align}\label{hom2}
{\rm{Aut}}(g_{\LB})\lras {\rm{GL}}(H^0(Z,F)).
\end{align}
In general, this map will not be injective.

\subsection{Proof of Proposition \ref{prop-hom}}
For this, we first recall the following result on the structure of LeBrun twistor spaces.

\begin{proposition}\label{prop-LB10}
If $n\ge 3$, $\dim H^0(Z, F)=4$ holds. Further, if $\Psi:Z\to\CP^3$ denotes 
the rational map induced by the linear system $|F|$, we have the following.
(i) The base locus of $|F|$ consists of two smooth rational curves $C_1$ 
and $\ol C_1$, which are mapped to the boundary sphere 
$\partial \H^3\subset n\#\CP^2$ by the twistor fibration 
$Z\to n\#\CP^2$.
(ii) The image $\Psi(Z)$ is a non-singular quadratic surface 
$\CP^1\times\CP^1$. (iii) If $Z'\to Z$ denotes the blow-up 
of $Z$ at $C_1\cup \ol C_1$, the composition 
$Z'\to Z\to \CP^1\times\CP^1$ is holomorphic. 
Further, the discriminant locus consists of $n$ smooth rational curves 
$\mathcal C_1,\dots,\mathcal C_n$ of bidegree $(1,1)$, which  
canonically correspond to the monopole points $p_1,\dots,p_n$.
\end{proposition}
\begin{proof}
We first take any smooth member $S\in |F|$ and consider an exact sequence
\begin{align}
0\lras \mathscr O_{Z}\lras F\lras K_{S}^{-1}\lras 0
\end{align}
and use $H^{1}(\mathscr O_{Z})=0$ to conclude that 
$\dim H^{0}(F)=1+\dim H^{0}(K_{S}^{-1})$ and ${\rm Bs}\,|F|={\rm{Bs}}\,|K_{S}^{-1}|$.
Since $S$ is obtained from $\CP^{1}\times\CP^{1}$ by 
blowing-up $n$ points lying on a curve of bidegree $(1,0)$
and also $n$ points lying on another curve of the same bidegree,
we readily obtain $\dim H^0(K_S^{-1})=3$.
We also obtain that  ${\rm{Bs}}\,|K_{S}^{-1}|$  is exactly the strict transform
of the last two curves, for which we write $C_1$ and $\overline C_1$.
(Note that to conclude these, we have used the assumption $n\ge 3$.)
As $\CC^*$ acts on $S$ fixing any points on $C_1\cup\overline C_1$
and the twistor fibration $Z\to n\#\CP^2$ is $\U(1)$-equivariant,
the image of $C_1$ under the twistor fibration must be the unique 2-sphere 
fixed by the $\U(1)$-action on $n\#\CP^2$.
Thus we obtain (i).
For (ii),  there are two distinguished pencils of degree-one divisors, 
which form a conjugate pair.
These two pencils generate a 3-dimensional system in $|F|$.
As $\dim |F|=3$, this means $|F|$ is in fact generated by the two pencils.
This implies that $\Psi(Z)$ is a smooth quadric.
For the first part of (iii), it suffices to notice that the union of 
 the base locus of the above 2 pencils (of degree-one divisors) are 
exactly $C_1\cup\overline{C}_1$, and they are eliminated after 
blowing-up  $C_1\cup\overline{C}_1$.
See \cite[\S 7]{LeBrun1991}, \cite[\S 3]{Poon1992} and \cite[\S 3]{KreusslerKurke} 
for details.
\end{proof}
\begin{remark}\label{rmk-ss}
{\em
The proposition is true for arbitrary $n\ge 0$ if we consider 
$H^0(Z,F)^{{\rm{U(1)}}}$ (the subspace consisting of 
all $\U(1)$-invariant sections) 
instead of  $H^0(Z,F)$, where $\U(1)$ is the subgroup of
fiber rotations of ${\rm{Aut}}(g_{\LB})$ coming from the bundle 
construction.}
\end{remark}

\begin{lemma}
\label{deg1}
Let $\Psi:Z\to\CP^3$ and $\mathcal C_1,\dots,\mathcal C_n$  
be as in Proposition \ref{prop-LB10}.
Then the following are all degree-one divisors on $Z$:
(i) the inverse images of curves on $\CP^1\times\CP^1$
whose bidegree are $(1,0)$ or $(0,1)$, (ii) the inverse images 
$\Psi^{-1}(\mathcal C_j)$, $1\le j\le n$.
\end{lemma}
\begin{proof}
If $D$ is a degree-one divisor, then $D+\ol{D}\in |F|$ holds by a 
Chern-class consideration (see \cite{Poon1992}).
Hence, since the rational map $\Psi$ is associated to $|F|$, 
any degree-one divisor is an irreducible component of a  reducible 
divisor of the form $\Psi^{-1}(H)$, where $H$ is a hyperplane 
in $\CP^3$. If the divisor $\Psi^{-1}(H)$ is reducible, then 
one of the following must clearly hold:
 $H\cap (\CP^1\times\CP^1)$ is reducible, or
 $H\cap (\CP^1\times\CP^1)$ is irreducible but 
$\Psi^{-1}(H)$ is reducible.
The former and latter  correspond to the cases (i) and (ii) 
in the lemma respectively.
\end{proof}
\begin{lemma}\label{lemma-inv1}
Suppose  $n\ge3$.
Then we have the following. (i) Any $\Phi\in {\rm{Aut}}(g_{\LB})$ leaves the boundary sphere $\partial\H^3$ (regarded as a subset of $n\#\CP^2$) invariant. 
(ii) Any $\Phi\in {\rm{Aut}}(g_{\LB})$ leaves the set of isolated fixed points invariant.
\end{lemma}
\begin{proof}
As before, we regard $\Phi$ as an automorphism of $Z$.
For (i), by Proposition \ref{prop-LB10} (i) it suffices to show  
$\{\Phi(C_1),\Phi(\ol C_1)\}=\{C_1,\ol C_1\}$. But since 
$C_1\cup\ol C_1$ are the base locus of the system $|F|$ as 
in Proposition \ref{prop-LB10} (i), this is automatic. 
For (ii), let $L_1,\dots, L_{n}$ be the twistor lines over the 
isolated fixed points of the $\U(1)$-action.
Then we have $\Psi(L_j)=\mathcal C_j$ and 
$\Psi^{-1}(\mathcal C_j)=D_j+\ol D_j$, where $D_j$ and $\ol D_j$ 
are degree one divisors intersecting transversally along $L_j$ 
(\cite[Proposition 3.6]{Poon1992}, \cite[\S 7]{LeBrun1991}).
Further, we have 
\begin{align}
\{\Phi(D_j),\Phi(\ol D_j)\set 1\le j\le n\}=\{D_j,\ol D_j\set 1\le j\le n\},
\end{align}
 since the $\mathcal C_j$-s are discriminant curves of the morphism 
$Z'\to\CP^1\times\CP^1$ by Proposition \ref{prop-LB10}.
Since $\Phi$ commutes with the real structure, 
this means that $\{\Phi(L_j)\set1\le j\le n\}=\{L_j\set 1\le j\le n\}$, which implies 
(ii) of the lemma.
\end{proof}
\begin{remark}
{\em
The lemma says that if $n\ge 3$, any $\Phi\in {\rm{Aut}}(g_{\LB})$ 
preserves the open subset $X_0$ (on which $\U(1)$ acts freely).
Obviously this does not hold if $n=0$ or $1$.
Namely, the general automorphism of the standard metrics on
$S^4$ or $\CP^2$ does not preserve the boundary sphere 
$\partial\H^3$. We will show in the 
next subsection that the lemma also fails to hold when $n=2$. 
}
\end{remark}

By Proposition \ref{prop-LB10}, when $n\ge 3$ we obtain a homomorphism
\begin{align}\label{hom5}
{\rm{Aut}}(g_{\LB})\lras {\rm{Aut}}(\CP^1\times\CP^1).
\end{align}
Further, by LeBrun's construction \cite{LeBrun1991}, the image quadric 
$\CP^1\times\CP^1$ can be regarded as a quotient space 
of the twistor space by a $\CC^*$-action, where the last action 
is the complexification of the semi-free $\U(1)$-action on $Z$.
More intrinsically, $\CP^1\times\CP^1$ can be 
interpreted as the minitwistor space (in the sense of Hitchin \cite{Hitchin}) 
of the hyperbolic space $\H^3$.
This in particular means that $\H^3$ can be canonically 
identified with the space of minitwistor lines in 
$\CP^1\times\CP^1$.
Such lines are explicitly given as  real irreducible
curves of bidegree $(1,1)$  which are disjoint from 
$(\CP^1\times\CP^1)^{\sigma}$ 
(the real locus on $\CP^1\times\CP^1$).
Furthermore, as a consequence of 
\begin{align}
\Psi^{-1}((\CP^1\times\CP^1)^{\sigma})=\bigcup_{x\in \partial\H^3}L_x,
\end{align}
where $L_x$ denotes the twistor line over a point $x\in n\#\CP^2$,
there is a natural identification 
$(\CP^1\times\CP^1)^{\sigma}\simeq\partial\H^3$. 
By Lemma \ref{lemma-inv1}, we have 
$\Phi(\partial \H^3)=\partial \H^3$ (on $n\#\CP^2$).
From this it follows that the automorphism of 
$\CP^1\times\CP^1$ coming from any 
$\Phi\in {\rm{Aut}}(g_{\LB})$ (via \eqref{hom5}) maps 
real $(1,1)$-curves disjoint from 
$(\CP^1\times\CP^1)^{\sigma}$ to
real $(1,1)$-curves disjoint from 
$(\CP^1\times\CP^1)^{\sigma}$.
Hence it maps minitwistor lines to minitwistor lines.
This way, we obtain a homomorphism
\begin{align}\label{hom4}
\rho:{\rm{Aut}}(g_{\LB})\lras {\rm{Aut}}(\H^3).
\end{align}
Moreover, since the action of Aut($g_{\LB}$) on 
$\CP^1\times\CP^1$ preserves  
$\mathcal C_1,\dots,\mathcal C_n$ (as they are discriminant curves), 
the image of \eqref{hom4} is contained in 
${\rm{Aut}}(\H^3;p_1,\dots,p_n)$.

To finish the proof of Proposition \ref{prop-hom}, 
it remains to show that 
if $\Phi\in {\rm{Aut}}(g_{\LB})$ is one of the  
lifts (obtained in Proposition \ref{hypconf}) of some 
$\phi\in {\rm{Aut}}(\H^3;p_1,\dots,p_n)$,  then 
$\rho(\Phi)=\phi$ holds. Take any point 
$p\in\H^3 \setminus \{p_1,\dots,p_n\}$, and put $q=\phi(p)$.
Let $\tilde p\in X_0$ be any point belonging to the fiber over 
$p$ and let $\tilde q=\Phi(\tilde p)$.
Let $L_{\tilde p}$ and $L_{\tilde q}$ be the twistor lines 
over $\tilde p$ and $\tilde q$, respectively.
Letting $\Phi$ also denote the induced automorphism on 
$\CP^1\times\CP^1$, we have 
$\Phi(\Psi(L_{\tilde p}))=\Psi(L_{\tilde q})$ by construction.
By the result of Jones-Tod \cite{JonesTod}
 on the relation between Penrose correspondence (between self-dual 
4-manifolds and 3-dimensional twistor spaces) and Hitchin correspondence 
(between Einstein-Weyl 3-manifolds and minitwistor spaces), 
the points on $\H^3$ which correspond to  
the minitwistor lines $\Psi(L_{\tilde p})$ and $\Psi(L_{\tilde q})$
are exactly $p$ and $q$ respectively. 
This implies $(\rho(\Phi))(p)=q$, as required. This completes the 
proof of Proposition \ref{prop-hom}.
\section{Poon's projective model}
\label{2CP2}
In this section, we determine the group of all conformal isometries of 
Poon's metrics on $2\#\CP^2$.
Although Poon's metrics can be constructed by LeBrun's hyperbolic ansatz,  
it turns out that, in contrast to the case $n\ge 3$, not all conformal 
isometries come from isometries of $\H^3$.
More precisely, we show that  such lifts form a subgroup of 
index $2$ in the full conformal isometry group.
\subsection{Automorphism group of Poon's projective models}
\label{ss:autoproj}
In order to analyze the automorphism group in the case of $2\#\CP^2$, 
instead of LeBrun's projective model, it is more convenient to use 
Poon's projective model of the twistor spaces (these are of course equivalent, 
see \cite[Section 7]{LeBrun1991}). In this subsection we investigate 
the holomorphic automorphism group of the projective models. 
We begin with recalling the following result due to Poon \cite{Poon1986}.
\begin{proposition}[\cite{Poon1986}] 
\label{prop-poon10}
Let $g$ be a self-dual metric on $2\#\CP^2$ of positive scalar 
curvature and $Z$ the twistor space of $g$. Then (i) the linear system 
$|F|$ is base point free, $5$-dimensional, and its associated morphism 
$\Psi:Z\to\CP^5$ is bimeromorphic to its image. 
(ii) The image $\tilde{Z}:=\Psi(Z)$ is an intersection of the two 
hyperquadrics in $\CP^5$  defined by 
\begin{equation}\label{Q12}
Q_{\infty}=\{w_0w_1+z_2^2+z_3^2+w_4w_5=0\},\,\,
Q_0=\{2w_0w_1+\lambda z_2^2+(3/2)z_3^2+w_4w_5=0\}
\end{equation}
where $(w_0,w_1,z_2,z_3,w_4,w_5)$  is a homogeneous coordinate on 
$\CP^5$ and $\lambda$ is a real number satisfying $3/2<\lambda<2$.
(iii) The singular locus of  $\tilde{Z}$ consists of 4 points 
\begin{align}
\begin{split}
&P_1:=(1,0,0,0,0,0),\,\,\ol P_1:=(0,1,0,0,0,0), \\
&P_3:=(0,0,0,0,1,0), \,\, \ol P_3:=(0,0,0,0,0,1),
\end{split}
\end{align}
and all these are ordinary nodes.
(iv) The morphism $\Psi:Z\to \tilde{Z}$ is a small resolution of these $4$ nodes.
(v) The real structure on $\CP^5$  induced by that on $Z$ is given by
\begin{equation}\label{rs1}
(w_0,w_1,z_2,z_3,w_4,w_5)\mapsto(\ol w_1,\ol w_0,\ol z_2,-\ol z_3,-\ol w_5,-\ol w_4).
\end{equation}
\end{proposition}

The identity component of the conformal transformation group of 
Poon's conformal class is ${\rm{U(1)}}\times \rm{U(1)}$. 
Correspondingly, the identity component of holomorphic 
transformation group of Poon's twistor space is $\CC^*\times\CC^*$.
In the above coordinates, this action is explicitly given by
\begin{equation}\label{G-action1}
(w_0,w_1,z_2,z_3,w_4,w_5)\mapsto(sw_0,s^{-1}w_1,z_2,z_3,tw_4,t^{-1}w_5),
\,\,(s,t)\in\CC^*\times\CC^*,
\end{equation}
which preserves the quadrics $Q_{\infty}$ and $Q_0$.
The map \eqref{G-action1} commutes  with the real structure \eqref{rs1}
if and only if $|s|=|t|=1$.

In the following we put $K={\rm{U(1)}}\times \rm{U(1)}$, and 
$G=\CC^*\times\CC^*$ for simplicity.
The $K$-action on $2\#\CP^2$ has exactly 4 fixed points.
Correspondingly, there are four $G$-invariant twistor lines in $Z$.
\begin{definition}{\em
Define the two real numbers
$\alpha:=\sqrt{4-2\lambda}$ and $\beta:=\sqrt{2\lambda-2}$.}
\end{definition} 
We remark that since $3/2<\lambda < 2$, we have the inequalities $0 < \alpha<\beta$.
These numbers will play an important role in the following.  
\begin{lemma}\label{lemma-inv7}
(i) Any $\Phi\in {\rm{Aut}}(g)$ leaves the set of four $K$-fixed points 
(on $2\#\CP^2$) invariant.
(ii) The image under $\Psi$ of the twistor lines over these $4$ points  
are conics whose equations are respectively given by
\begin{equation}\label{tl1}
\{\alpha z_2- \i z_3=w_4=w_5=w_0w_1+(2\lambda-3)z_2^2=0\},
\end{equation}
\begin{equation}\label{tl2}
\{\alpha z_2+ \i z_3=w_4=w_5=w_0w_1+(2\lambda-3)z_2^2=0\},
\end{equation}
\begin{equation}\label{tl3}
\{\beta z_2- \i z_3=w_0=w_1=(3-2\lambda)z_2^2+w_4w_5=0\},
\end{equation}
\begin{equation}\label{tl4}
\{\beta z_2+\i z_3=w_0=
w_1=(3-2\lambda)z_2^2+w_4w_5=0\}.
\end{equation}
\end{lemma}
\begin{proof}
For (i), consider the two linear projections 
$f_j:\CP^5\to\CP^3$ ($j=1,3$) defined by
\begin{align}
\label{proj1}
f_1(w_0,w_1,z_2,z_3,w_4,w_5)&=(z_2,z_3,w_4,w_5),\\
\label{proj2}
f_3(w_0,w_1,z_2,z_3,w_4,w_5)&=(w_0,w_1,z_2,z_3).
\end{align}
By an elementary computation, we have
\begin{align}
\label{image1}
f_1(\tilde{Z})=\{\alpha^2 z_2^2+z_3^2+2w_4w_5=0\},\,\,
f_3(\tilde{Z})=\{2w_0w_1+\beta^2z_2^2+z_3^2=0\}.
\end{align}
Intrinsically, the composition $f_j\circ\Psi:Z\to\CP^3$ 
is the meromorphic map associated to the linear system corresponding 
to the subspace $H^0(Z, F)^{G_j}$, where $G_1$ and $G_3$ are 
$\CC^*$-subgroups of $G$ defined by 
\begin{equation}\label{G1}
G_1=\{{\rm{diag}}(s,s^{-1},1,1,1,1)\in {\rm{PGL}}(6,\CC)\set s\in \CC^*\}
\end{equation} and
\begin{equation} \label{G3}
G_3=\{{\rm{diag}}(1,1,1,1,t,t^{-1})\in {\rm{PGL}}(6,\CC)\set t\in \CC^*\}.
\end{equation}
Since $\alpha\beta\neq 0$ by Poon's constraint $(3/2)<\lambda<2$, 
\eqref{image1} means that the images $f_1(\tilde{Z})$ and $f_3(\tilde{Z})$ are 
non-singular quadrics.
Hence both are isomorphic to a product $\CP^1\times\CP^1$. 
(Both of these two rational maps from $Z$ to $\CP^1\times\CP^1$ 
exactly correspond to the map $\Psi:Z\to\CP^1\times\CP^1$ 
for LeBrun twistor spaces considered above for $n \geq 3$).
Then by taking the pull-back of pencils on 
$\CP^1\times\CP^1$ of bidegree 
$(1,0)$ and $(0,1)$, we obtain 2 pencils on $Z$ for each of $j=1$ and $j=3$.
Hence we obtain $4$ pencils on $Z$ in total.
Since $(f_j\circ\Psi)^*\mathcal O(1)\simeq F$ 
and hyperplane sections of the quadrics are bidegree $(1,1)$, 
members of the 4 pencils are degree one, 
since the intersection number of the divisor with twistor lines is one.
On the other hand, by \cite[Lemma 1.9]{Poon1992}, for $2\#\CP^2$  
there are at most $4$ degree one line bundles on $Z$ which have a 
non-trivial section.
Further, since $\dim |D|\le 1$ for any degree 1 divisor $D$ on any 
twistor space on $n\#\mathbb{CP}^2$ by \cite[Lemma 1.10 (2)]{Poon1992},
these 4 pencils have mutually different Chern classes.
This implies that there are no pencils of degree 
one other than the above $4$ ones.
Obviously, the $G$-action preserves each of these pencils.
Furthermore, it can be readily seen by 
\eqref{G-action1}, \eqref{proj1}, \eqref{proj2}, and \eqref{image1} that 
$G$ acts non-trivially on the parameter space ($\CP^1$)  of the pencils.
Hence each pencil has precisely two $G$-invariant members, 
so that we have eight $G$-invariant degree one divisors in total.
By \eqref{rs1}, it is clear that the two $G$-invariant divisors 
in the same pencil form a conjugate pair. So we may write 
$\{D_j,\ol D_j\set 1\le j\le 4\}$ for the set of $G$-invariant 
degree one divisors.

We next compute the defining equations of the images of 
these $8$ divisors in $\CP^5$ (under $\Psi$) 
in the following way.
First, by using \eqref{image1}, we can obtain equations of  the 
four $G$-invariant curves of bidegrees $(1,0)$ or $(0,1)$.
(For instance, one of them is given by $\alpha z_2- \i z_3=w_4=0$.)
Next, substituting the equations into  (any one of)  \eqref{Q12}, 
we obtain the equations of the images $\Psi(D_j)$ and $\Psi(\ol{D}_j)$.
(For the above curve,  the equations become 
$\alpha z_2- \i z_3=w_4=w_0w_1+(2\lambda-3)z_2^2=0$.)
The last equations imply that $\Phi(D_j)$ is a quadratic cone 
in $\CP^3$ and that its vertex is exactly one of the 
four singular points of $\tilde{Z}=\Psi(Z)$.
(For the above case $\ol{P}_3$ is contained as the vertex.)
Recall that $\Psi$ is the morphism which contracts the 
four rational curves, and that the images of the curves are exactly 
the singular points of $\tilde{Z}$.
On the other hand, by (\cite[Lemma 1.9]{Poon1992}), any 
degree-one divisor is non-singular.
Therefore  the morphisms $D_j\to\Psi(D_j)$ and $\ol{D}_j\to\Psi(\ol D_j)$
factor through the minimal resolution  of the quadratic cones.
Then again by (\cite[Lemma 1.10]{Poon1992}), $D_j$ and $\ol D_j$ 
are obtained from 
$\Sigma_2=\mathbb P(\mathcal O\oplus \mathcal O(2))$ 
(the minimal resolution of the cone) by blowing-up one point.

In a similar fashion, we can compute the defining equations of $\Psi(D)$
for a non-$G$-invariant degree-one divisor $D$.
 (For instance, one of them is given by
\begin{align} 
w_4-c(\alpha z_2- \i z_3)=\alpha z_2+ \i z_3+2cw_5=
 2w_0w_1+(2\lambda-2)z_2^2+z_3^2=0,
\end{align}
where $c\in \CC^*$.)
From these (and also by the constraint $3/2<\lambda<2$), 
we obtain that $\Psi(D)$ is biholomorphic to a non-singular 
quadric in $\CP^3$; namely $\CP^1\times\CP^1$.
Then again by \cite[Lemmas 1.9 and 1.10]{Poon1992} we obtain that 
the divisor $D$ is obtained from $\CP^1\times\CP^1$ 
by blowing-up one point.
Since the one point blow-up of $\Sigma_2$ and that of 
$\CP^1\times\CP^1$ cannot be biholomorphic, 
we conclude that the $G$-invariant divisors $D_j, \ol D_j$ 
and non-$G$-invariant divisors $D$ cannot be biholomorphic.

For a given  $\Phi\in{\rm{Aut}}(g)$, if we use the same letter to 
denote the induced automorphism of $Z$, $\Phi$ clearly 
preserves the set of $4$ pencils (as any $\Phi\in{\rm{Aut}}(Z)$ 
preserves the degree of divisors).
Further, by the above distinction of complex structure between 
$G$-invariant and non-$G$-invariant members, the set of 
$G$-invariant members (which are explicitly given by 
$\{D_j,\ol D_j\set 1\le j\le 4\}$) are preserved under $\Phi$.
As $\Phi$ preserves the real structure, this means that 
$\Phi$ preserves the set $\{D_j\cap\ol D_j\set 1\le j\le 4\}$. 
Since these are exactly the set of $G$-invariant twistor lines,
this implies the claim (i)  of the lemma.

For (ii) we notice that each $D_j+\ol D_j$ is contracted to a 
reducible curve of bidegree $(1,1)$ under precisely one of the 
two rational maps $f_1\circ\Psi$ and $f_3\circ\Psi$.
Therefore each twistor line $L_j=D_j\cap \ol D_j$ is mapped 
to a real $G$-fixed point on (one of) the image quadrics.
On the quadric $f_1(\tilde{Z})\simeq\CP^1\times\CP^1$, 
there are exactly two real $G$-fixed points, and they are 
explicitly given by 
\begin{align*}
\{\alpha z_2- \i z_3=w_4=w_5=0\}, \mbox{ and } 
\{\alpha z_2+\i z_3=w_4=w_5=0\}.
\end{align*}
Similarly, on the quadric $f_3(\tilde{Z})$, real $G$-fixed points 
are explicitly given by 
\begin{align*}
\{\beta z_2-\i z_3=w_0=w_1=0\}, \mbox{ and } 
\{\beta z_2+ \i z_3=w_0=w_1=0\}.
\end{align*}
Computing the inverse images of these 4 points under $f_1$ and $f_3$ 
(namely substituting these into the equations \eqref{Q12}), 
we obtain the desired equations
for the images of $G$-invariant twistor lines.
\end{proof}

The homomorphism \eqref{hom2} and the coordinates 
$(w_0,w_1,z_2,z_3,w_4,w_5)$ give a homomorphism
\begin{align}\label{hom3}
{\rm{Aut}}(g)\lras {\rm{GL}}(6,\CC).
\end{align} 
We shall obtain the image of \eqref{hom3} explicitly.
Take any $\Phi\in{\rm{Aut}}(g)$ and let   
$U\in{\rm{GL}}(6,\CC)$ be its image.
Then as in the case of $n\ge3$, $U$ preserves the variety $\tilde{Z}$.
Hence $U$ preserves the singular set $\{P_1,P_3,\ol P_1,\ol P_3\}$.
Taking the real structure into account,  
the following two possibilities can occur:

\vspace{2mm}
(I) $\{U(P_1), U(\ol P_1)\}=\{P_1,\ol P_1\}$ and  
$\{U(P_3), U(\ol P_3)\}=\{P_3,\ol P_3\}$,

(II) $\{U(P_1), U(\ol P_1)\}=\{P_3,\ol P_3\}$ and 
$\{U(P_3), U(\ol P_3)\}=\{P_1,\ol P_1\}$.

\vspace{2mm}
\noindent
For case (I), using the fact that $U$ commutes 
with the real structure \eqref{rs1}, it is easy to deduce that 
$U$ is of the form
\begin{eqnarray}\label{U1}
\left(
\begin{array}{ccc}
A_{11}&A_{12} &O \\
O&A_{22} &O \\
O& A_{32}& A_{33}\\
\end{array} 
\right),
\end{eqnarray}
where $A_{12},A_{22}$ and $A_{32}$ are $2\times 2$ 
matrices with $\det A_{22}\neq 0$ and
\begin{eqnarray}\label{ent31}
A_{11}=\left( 
\begin{array}{cc}
a &0 \\
0 &\ol a \\
\end{array} 
\right){\text{ or }}
\left( 
\begin{array}{cc}
0 &a \\
\ol a &0 \\
\end{array} 
\right),\,\,
A_{33}=\left( 
\begin{array}{cc}
b &0 \\
0 &\ol b \\
\end{array} 
\right){\text{ or }}
\left( 
\begin{array}{cc}
0 &b \\
\ol b &0 \\
\end{array} 
\right),
\end{eqnarray}
where $a,b\in\CC^*$.
Similarly, for case (II), $U$ is of the form
\begin{eqnarray}\label{U2}
\left(
\begin{array}{ccc}
O&A_{12} &A_{13} \\
O&A_{22} &O \\
A_{31}& A_{32}& O\\
\end{array} 
\right),
\end{eqnarray}
where $A_{12},A_{22}$ and $A_{32}$ are $2\times 2$ 
matrices with $\det A_{22}\neq 0$ and
\begin{eqnarray}\label{ent32}
A_{13}=\left( 
\begin{array}{cc}
a &0 \\
0 &-\ol a \\
\end{array} 
\right){\text{ or }}
\left( 
\begin{array}{cc}
0 &a \\
-\ol a &0 \\
\end{array} 
\right),\,\,
A_{31}=\left( 
\begin{array}{cc}
b&0 \\
0 &-\ol b \\
\end{array} 
\right){\text{ or }}
\left( 
\begin{array}{cc}
0 &b \\
-\ol b &0 \\
\end{array} 
\right)
\end{eqnarray}
where $a,b\in\CC^*$.

Using Lemma \ref{lemma-inv7}, we can deduce another 
restriction for the $6\times 6$ matrix $U$ as follows.

\begin{lemma}\label{lemma-zero}
(i) In the presentations \eqref{U1} and \eqref{U2},  $A_{12}=A_{32}=O$ holds.
(ii) If $U$ belongs to the case (I), the matrix $A_{22}$ must be of the form
\begin{equation}
\label{ent33}
A_{22}=
c\left(
\begin{array}{cc}
     1&0   \\
     0&  1 \\
\end{array}
\right)
\,\,\,{\text{ or }}\,\,\,
c\left(
\begin{array}{cc}
     1&0   \\
     0&  -1 \\
\end{array}
\right),\,\,c\in\mathbb R^*.
\end{equation}
(iii) If $U$ belongs to the case (II), we have
\begin{equation}
\label{ent34}
A_{22}=
c\left(
\begin{array}{cc}
     0&1   \\
     \alpha\beta&  0 \\
\end{array}
\right)
\,\,\,{\text{ or }}\,\,\,
c\left(
\begin{array}{cc}
     0&-1   \\
     \alpha\beta&  0 \\
\end{array}
\right),\,\,c\in \i \mathbb R^*.
\end{equation}

\end{lemma}
\begin{proof}
First, we note that by Lemma \ref{lemma-inv7}, $U$ has to leave the 
set of 4 conics \eqref{tl1}--\eqref{tl4} invariant.
In the case (I), namely if $\{U(P_1),U(\ol{P}_1)\}=\{P_1,\ol P_1\}$, 
the set of the two conics $\{\eqref{tl1},\eqref{tl2}\}$ must be 
preserved under $U$, since \eqref{tl1} and \eqref{tl2} contain 
$P_1$ and $\ol P_1$, and \eqref{tl3} and \eqref{tl4} do not.
Similarly, the set $\{\eqref{tl3},\eqref{tl4}\}$ must also be 
preserved under $U$.

A generic point on the conics \eqref{tl1} and \eqref{tl2} is of the form
$(w_0,w_1,1,\mp \i \alpha,0,0)$.
Since 
\begin{equation}\label{mat1}
\left(
\begin{array}{ccc}
A_{11} & A_{12}  & O  \\
O  & A_{22}  &  O \\
O  & A_{32}  & A_{33}
\end{array}
\right)
\left(
\begin{array}{c}
  w_0  \\
  w_1\\
  1\\
  \mp \i \alpha\\
  0\\
  0    
\end{array}
\right)
=\left(
\begin{array}{c}
  \ast\\
  \ast\\
  A_{22}\left(
  \begin{array}
 {c}
  1\\
  \mp \i \alpha
  \end{array}\right)
  \\
  A_{32}\left(
  \begin{array}
 {c}
  1\\
  \mp \i \alpha
  \end{array}\right)  
\end{array}\right),
\end{equation}
and these points still belong to \eqref{tl1} or \eqref{tl2}, we obtain 
$$
A_{32}
\left(
\begin{array}{c}
  1  \\
  -\i\alpha    
\end{array}
\right)
=
A_{32}
\left(
\begin{array}{c}
  1  \\
  \i\alpha    
\end{array}
\right)
=
\left(
\begin{array}{c}
  0  \\
  0    
\end{array}
\right).
$$
Since $\alpha=\sqrt{4-2\lambda}\neq0$, we obtain $A_{32}=0$.
Similarly, considering the analogous  requirement for 
\eqref{tl3} and \eqref{tl4}, we obtain $A_{12}=0$. 

Thus we have obtained the claim (i) for the case (I).
For the case (II), namely if  
$\{U(P_1),U(\ol{P}_1)\}=\{P_3,\ol P_3\}$, the sets of the two conics 
$\{$\eqref{tl1}, \eqref{tl2}$\}$ and $\{$\eqref{tl3}, \eqref{tl4}$\}$ 
must be interchanged under $U$.
From this we can again deduce $A_{12}=A_{32}=O$  by similar computations.
Hence we obtain (i).

Next we show (ii).
Suppose $U$ belongs to the case (I).
Then since the right hand side of  \eqref{mat1} belongs to 
the conics \eqref{tl1} or \eqref{tl2}, as points on 
$\CP^1$ we have either 
\begin{equation}\label{mat5}
A_{22}\left(
  \begin{array}
 {c}
  1\\
  \i\alpha
  \end{array}\right)
  =
  \left(
  \begin{array}
 {c}
  1\\
  -\i\alpha
  \end{array}\right)
  \,\,{\text{ and }}\,\,
  A_{22}\left(
  \begin{array}
 {c}
  1\\
  -\i\alpha
  \end{array}\right)
  =
  \left(
  \begin{array}
 {c}
  1\\
  \i\alpha
  \end{array}\right)
\end{equation}
or
\begin{equation}\label{mat6}
A_{22}\left(
  \begin{array}
 {c}
  1\\
  \i\alpha
  \end{array}\right)
  =
  \left(
  \begin{array}
 {c}
  1\\
  \i\alpha
  \end{array}\right)
  \,\,{\text{ and }}\,\,
  A_{22}\left(
  \begin{array}
 {c}
  1\\
  -\i\alpha
  \end{array}\right)
  =
  \left(
  \begin{array}
 {c}
  1\\
  -\i\alpha
  \end{array}\right),
\end{equation}
according to whether $U$ interchanges \eqref{tl1} and \eqref{tl2} or not.
Similarly, by using the computations  to deduce $A_{12}=0$, either 
\begin{equation}\label{mat7}
A_{22}\left(
  \begin{array}
 {c}
  1\\
  \i\beta
  \end{array}\right)
  =
  \left(
  \begin{array}
 {c}
  1\\
  -\i\beta
  \end{array}\right)
  \,\,{\text{ and }}\,\,
  A_{22}\left(
  \begin{array}
 {c}
  1\\
  -\i\beta
  \end{array}\right)
  =
  \left(
  \begin{array}
 {c}
  1\\
  \i\beta
  \end{array}\right)
\end{equation}
or
\begin{equation}\label{mat8}
A_{22}\left(
  \begin{array}
 {c}
  1\\
  \i\beta
  \end{array}\right)
  =
  \left(
  \begin{array}
 {c}
  1\\
  \i\beta
  \end{array}\right)
  \,\,{\text{ and }}\,\,
  A_{22}\left(
  \begin{array}
 {c}
  1\\
  -\i\beta
  \end{array}\right)
  =
  \left(
  \begin{array}
 {c}
  1\\
  -\i\beta
  \end{array}\right),
\end{equation}
according to whether $U$ interchanges \eqref{tl3} and \eqref{tl4} or not.
We note that as points on $\CP^1$
\begin{equation}\label{4pt}
(1,\i\alpha),\,\,(1,-\i\alpha),\,\,(1,\i\beta),\,\,(1,-\i\beta)
\end{equation}
are four distinct points.
Thus \eqref{mat5}--\eqref{mat8} mean that in any case the 
projective transformation determined by the matrix $A_{22}$ 
leaves the set of the 4 points \eqref{4pt} invariant.
If \eqref{mat6} and \eqref{mat8} happens, then $A_{22}$ fixes all $4$ 
points. This means $A_{22}=cI_2$ for some $c\in\CC^*$. 
If \eqref{mat5} and \eqref{mat7} happen, then $A_{22}$ interchanges 
$(1,\i\alpha)$ and $(1,-\i\alpha)$ and also $(1,\i\beta)$ and $(1,-\i\beta)$.
This means $A_{22}=c\,{\text{diag}}(1,-1)$. 
A simple computation also shows that there exists no projective 
transformation realizing the remaining two cases.
Moreover, since $U$ commutes with the real structure \eqref{rs1}, 
we readily obtain $c\in \mathbb R$. Thus we obtain the claim (ii)
in case (I).

If $U$ belongs to the case (II), by similar computation as above, 
we deduce that, as a projective transformation, $A_{22}$ maps 
$(1,\i\alpha)$ to either $(1,\i\beta)$ or $(1,-\i\beta)$ 
(so that $(1,-\i\alpha)$ is mapped to $(1,-\i\beta)$ or 
$(1,\i\beta)$ respectively).
 Further, $A_{22}$ maps $(1,\i\beta)$ to either 
$(1,\i\alpha)$ or $(1,-\i\alpha)$
 (so that $(1,-\i\beta)$ is mapped to $(1,-\i\alpha)$ or 
$(1,\i\alpha)$ respectively).
Among these $2\cdot2=4$ possibilities, only the two cases 
\begin{align*}
A_{22}:(1,\i\alpha)\mapsto (1,\i\beta)\,\,
{\text{ and }}\,\,(1,\i\beta)\mapsto (1,\i\alpha),
\end{align*} 
and 
\begin{align*}
A_{22}:(1,\i\alpha)\mapsto (1,-\i\beta)\,\,
{\text{ and }}\,\,(1,\i\beta)\mapsto (1,-\i\alpha)
\end{align*}
can actually occur, and for each these cases $A_{22}$ is represented by the matrices
\begin{equation}
A_{22}=
\left(
\begin{array}{cc}
     0&1   \\
     \alpha\beta&  0 \\
\end{array}
\right)
\,\,\,{\text{ and }}\,\,\,
\left(
\begin{array}{cc}
     0&-1   \\
     \alpha\beta&  0 \\
\end{array}
\right),
\end{equation}
respectively.
Finally, again by commutativity with \eqref{rs1} we obtain $c\in \i\mathbb R$.
This completes the proof of claim (iii).
\end{proof}

The next lemma determines all automorphisms of the projective model 
$\tilde{Z}$ which commute with the real structure.

\begin{lemma}\label{lemma:autproj}
 (i) Let $U$ be a $6 \times 6$ matrix in Case (I) of the form \eqref{U1}, 
where $A_{11}$ and $A_{33}$ are 
as in \eqref{ent31} and $A_{12}=A_{32}=O$ and $A_{22}$ is
as in \eqref{ent33} by Lemma \ref{lemma-zero}.
Then after normalizing by a scalar multiplication to make
$c=1$, $U$ preserves the projective model $\tilde{Z}$ if and only if 
the entries in \eqref{ent31} satisfy $|a|=|b|=1$.
(ii) Let $U$ be the $6 \times 6$ matrix in Case (II) of the form \eqref{U2}, 
where $A_{13}$ and $A_{31}$ are as in 
\eqref{ent32} and $A_{12}=A_{32}=O$ and $A_{22}$ is as in 
\eqref{ent34} by Lemma \ref{lemma-zero}.
Then after normalizing by a scalar multiplication to make
$c=\i$, $U$ preserves the projective model $\tilde{Z}$  
if and only if the entries in \eqref{ent32} satisfy
\label{l1}
\begin{align}
\label{newones}
|a|=\beta, \,\,|b| =\alpha \,\,\,({\text{and }}c=\i).
\end{align}
\end{lemma}
\begin{proof}
 We only show  (ii) since (i) can be proved by a similar (and simpler) 
computation. We recall that $\tilde{Z}$ is defined by the following $2$ quadratic 
polynomials:
\begin{align}
h_0 &= 2 w_0 w_1 + \lambda z_2^2 + \frac{3}{2} z_3^2 + w_4 w_5,\\
h_{\infty} &= w_0 w_1 + z_2^2 + z_3^2 + w_4 w_5.
\end{align}
We also recall $\alpha^2 = 4 - 2 \lambda,\, \beta^2 = 2 \lambda -2$. 
Let the constants $({a},{b},{c})$ be arbitrary satisfying  $c\in \i\mathbb R$.
Then by substitution, we obtain
\begin{align}
U h_0& = -2 |a|^2 w_4 w_5 + {c}^2 \lambda z_3^2 
      + (3/2) {c}^2 \alpha^2 \beta^2 z_2^2 - |{b}|^2 w_0 w_1,\\
U h_{\infty}& = -| {a}|^2 w_4 w_5 + {c}^2 z_3^2 
   + {c}^2 \alpha^2 \beta^2 z_2^2 - |{b}|^2 w_0 w_1. 
\end{align}
By multiplying a real constant to $U$, we may suppose $|b|=1$.
So we have constants $(a, b,c)$ with $|b|=1$ determining 
$U$ in Case (II).  
This gives, 
\begin{align}
U h_0 &= -2 |a|^2 w_4 w_5 + c^2 \lambda z_3^2 
      + (3/2) c^2 \alpha^2 \beta^2 z_2^2 - w_0 w_1,\\
U h_{\infty} &= -|a|^2 w_4 w_5 + c^2 z_3^2 
   + c^2 \alpha^2 \beta^2 z_2^2 - w_0 w_1. 
\end{align}
If $U$ preserves $\tilde{Z}$, then preserves the quadratic ideal $(h_0, h_{\infty})$, 
so there exist constants $c_1$ and $c_2$ so that
\begin{align}
\label{34}
c_1 U h_0 - c_2 U h_{\infty} 
= h_0 - h_{\infty} = w_0 w_1 + (\lambda-1) z_2^2 + (1/2) z_3^2.
\end{align}
A computation gives
\begin{align}
\begin{split}
\label{n}
c_1 U h_0 - c_2 U h_{\infty} &= -(c_1 - c_2) w_0 w_1 
    +|a|^2 (-2 c_1  + c_2 ) w_4 w_5  \\ 
&+ c^2 ( c_1 \lambda - c_2) z_3^2 
+ \left(\frac{3}{2}c_1 - c_2 \right) c^2 \alpha^2 \beta^2 z_2^2.
\end{split}
\end{align}
Comparing with \eqref{34}, we see that 
$c_2 - c_1 = 1$ and $|a|^2 ( -2 c_1 + c_2) = 0$.
Since $|a| \neq 0$,   we obtain 
\begin{align}
c_1 = 1, c_2 = 2. 
\end{align}
Then we have 
\begin{align}
U h_0 - 2 U h_{\infty} = w_0 w_1 + c^2(\lambda - 2) z_3^2 
                          - (1/2) c^2 \alpha^2 \beta^2 z_2^2.
\end{align}
Comparing coefficients with \eqref{34}, we have
\begin{align}
c^2(\lambda - 2) &= 1/2\label{345}\\
 - (1/2) c^2 \alpha^2 \beta^2 &= (\lambda - 1).\label{345-2}
\end{align}
By \eqref{345} we obtain $c^2=-\alpha^{-2}$.
Further, if this is satisfied, \eqref{345-2} automatically holds.
So we find that $h_0-h_{\infty}\in (Uh_0,Uh_{\infty})$ holds
if and only if after a rescaling the entries of $U$ satisfy 
$c=i$ and $|b|=\alpha$.

Next, by rescaling, we assume $|a|=1$.
We compute that 
\begin{align}
U h_0 &= -2 w_4 w_5 + c^2 \lambda z_3^2 + \frac{3}{2}c^2 \alpha^2
\beta^2 z_2^2 - |b|^2 w_0 w_1\\
U h_{\infty} &= - w_4 w_5 + c^2 z_3^2 + c^2 \alpha^2 \beta^2 z_2^2
- |b|^2 w_0 w_1
\end{align}
Consider the element 
\begin{align}
h_0 - 2 h_{\infty} = (\lambda -2) z_2^2 - \frac{1}{2} z_3^2 - w_4 w_5.
\end{align}
We next find $c_1$ and $c_2$ so that 
\begin{align}
c_1 U h_0 - c_2 U h_{\infty} =  (\lambda -2) z_2^2 - \frac{1}{2} z_3^2 - w_4 w_5.
\end{align}
We compute
\begin{align}
\begin{split}
c_1 U h_0 - c_2 U h_{\infty} &= (-2 c_1 + c_2) w_4 w_5 
+ c^2( c_1 \lambda - c_2) z_3^2 \\
&+ \left( \frac{3}{2} c_1 - c_2 \right)
c^2 \alpha^2 \beta^2 z_2^2 - |b|^2( c_1 - c_2) w_0 w_1. 
\end{split}
\end{align}
We find that 
\begin{align}
|b|^2 (c_1 - c_2) = 0, \   - 2 c_1 + c_2 = -1,
\end{align}
which implies that $c_1 = c_2 = 1$.  We then have
\begin{align}
\frac{1}{2} c^2 \alpha^2 \beta^2 = \lambda - 2, \ c^2( \lambda -1) = -\frac{1}{2}. 
\end{align}
The latter equation implies $c^2=-|\beta|^{-2}$, which implies the former
equation. So we find that $h_0-2h_{\infty}\in (Uh_0,Uh_{\infty})$ holds
if and only if after a rescaling the entries of $U$ satisfy 
$c=\i$ and $|a|=\beta$.

On the other hand, as  
$(h_0,h_{\infty})=(h_0-h_{\infty},h_0-2h_{\infty})$, 
and $(h_0,h_{\infty})=(Uh_0,Uh_{\infty})$ holds if and only if 
$h_0\in (Uh_0,Uh_{\infty})$ and $h_{\infty}\in (Uh_0,Uh_{\infty})$.
Hence by a combination of the above two, we conclude that 
$U$ preserves $\tilde{Z}$ if and only if $U$ can be rescaled to satisfy 
$c=\i$, $|a|=\beta$ and $|b|=\alpha$.
\end{proof}

\vspace{2mm}
According to Lemma \ref{lemma:autproj}, 
in Case (I), each of $A_{11}$, $A_{22}$ and $A_{33}$ has 2 types of choices, 
with $|a|=|b|=1$.
Hence the automorphisms in (i) of  Lemma \ref{lemma:autproj} 
constitute $2^3=8$ tori.
Similarly by Lemma \ref{lemma:autproj} (ii), the same is true for Case (II), 
so that we again obtain $8$ tori.
Thus we obtain $16$ tori in the holomorphic automorphism group of $\tilde{Z}$.
All automorphisms in these $16$ tori commute with the real structure.

\subsection{Determination of small resolutions}
\label{ss:detsr}
As in Proposition \ref{prop-poon10}, the projective model $\tilde Z$ of 
Poon's twistor spaces on $2\#\CP^2$ has precisely 4 ordinary nodes 
$P_1,\ol P_1,P_3$ and $\ol P_3$.
The actual twistor space $Z$ is obtained from $\tilde Z$ by taking small 
resolutions for each node.
Of course, there are exactly 2 ways of small resolutions for each node.
(We refer the reader to \cite[Section 12]{Kollar1987} for a 
discussion of the small resolutions of ordinary nodes
of threefolds.)
Since the resolution much preserve the real structure,  
the small resolutions of $P_1$ and $P_3$ uniquely determine those 
of $\ol P_1$ and $\ol P_3$ respectively, so there are exactly 
$4$ ways to obtain small resolutions of the variety $\tilde Z$ which 
preserve the real structure.
In this subsection we explicitly determine which small resolutions 
yield the twistor space.
This gives a completely explicit construction of the twistor spaces 
of Poon's metrics on $2\#\CP^2$, starting from his projective models in $\CP^5$.

For the purpose of specifying the small resolutions of ordinary nodes 
of $\tilde{Z}$, we first investigate local structure of $\tilde{Z}$ 
in  neighborhoods of the singularities.
First we take $P_1=(1,0,0,0,0,0)$ and  $\ol P_1=(0,1,0,0,0,0)$. 
By \eqref{tl1} and \eqref{tl2} the two irreducible components of 
the two reducible hyperplane sections $\tilde{Z}\cap H_{\alpha}$ and 
$\tilde{Z}\cap H_{-\alpha}$ contain $P_1$ and $\ol P_1$ as smooth points. 
Namely, the 4 surfaces 
\begin{align}
D'_1:=\{\alpha z_2-\i z_3=w_4=w_0w_1+(2\lambda-3)z_2^2=0\},\label{d1-1}\\
\ol D'_1:=\{\alpha z_2-\i z_3=w_5=w_0w_1+(2\lambda-3)z_2^2=0\},\label{d1-2}\\
D'_2:=\{\alpha z_2+\i z_3=w_4=w_0w_1+(2\lambda-3)z_2^2=0\},\label{d1-3}\\
\ol D'_2:=\{\alpha z_2+\i z_3=w_5=w_0w_1+(2\lambda-3)z_2^2=0\label{d1-4}\},
\end{align}
all of which are cones over a smooth conic,
share $P_1$ and $\ol P_1$ as smooth points.
Note that $\ol D'_1=\sigma(D'_1)$ and $\ol D'_2=\sigma(D'_2)$ hold, and 
that all of these 4 surfaces are $G$-invariant.
The configuration of these 4 surfaces and the ordinary nodes is 
illustrated in the diagram on the left in Figure \ref{fig-4faces}.
In a neighborhood of $P_1$, by setting $w_0=1$ in the defining 
equations in \eqref{Q12}  and eliminating $w_1$, we can think of 
$\tilde{Z}$ as defined in $\CC^4=\{(z_2,z_3,w_4,w_5)\}$ by the equation 
\begin{align}\label{X}
\alpha^2z_2^2+z_3^2+2w_4w_5=0,
\end{align}
from which one can see that $P_1$ is an ordinary node of $\tilde{Z}$.
Similarly, by neglecting the last common hyperquadric  in \eqref{d1-1}--\eqref{d1-4},
these 4 surfaces can be considered to be defined in the same $\CC^4$
(at least in a neighborhood of $P_1$).

By the equations \eqref{X} and \eqref{d1-1}--\eqref{d1-4} 
(with the last common quadratic equation neglected), a small 
resolution of $\tilde{Z}$ at $P_1$  is clearly specified by 
which pair among  $\{D'_1,\ol D'_2\}$ or $\{\ol D'_1,D'_2\}$ 
is blown-up at $P_1$.
By exchanging the role of $w_0$ and $w_1$ in the above argument, 
we see that a small resolution at the conjugate point $\ol P_1$ can 
also be specified by which pair of  $\{D'_1,\ol D'_2\}$ or 
$\{\ol D'_1,D'_2\}$ is blown-up at $\ol P_1$.

Similarly, by \eqref{tl3} and \eqref{tl4}, the other two reducible 
hyperplane sections $\tilde{Z}\cap H_{\beta}$ and $\tilde{Z}\cap H_{-\beta}$ 
contain $P_3$ and $\ol P_3$ as smooth points.
They consist of the four $G$-invariant surfaces
\begin{align}
D'_3:=\{\beta z_2-\i z_3=w_0=(3-2\lambda)z_2^2+w_4w_5=0\},\label{d1-7}\\
\ol D'_3:=\{\beta z_2-\i z_3=w_1=(3-2\lambda)z_2^2+w_4w_5=0\},\label{d1-8}\\
D'_4:=\{\beta z_2+\i z_3=w_0=(3-2\lambda)z_2^2+w_4w_5=0\},\label{d1-9}\\
\ol D'_4:=\{\beta z_2+\i z_3=w_1=(3-2\lambda)z_2^2+w_4w_5=0\label{d1-10}\}.
\end{align}
These are illustrated in the diagram on the right in Figure~\ref{fig-4faces}.
By the same reasons as for $P_1$ and $\ol P_1$, the small resolutions 
of $\tilde{Z}$ at $P_3$ and $\ol P_3$ are specified 
by which pair among $\{D'_3,\ol D'_4\}$ or $\{\ol D'_3,D'_4\}$ 
is blown-up at $P_3$ and $\ol P_3$ respectively.

\begin{figure}
\includegraphics{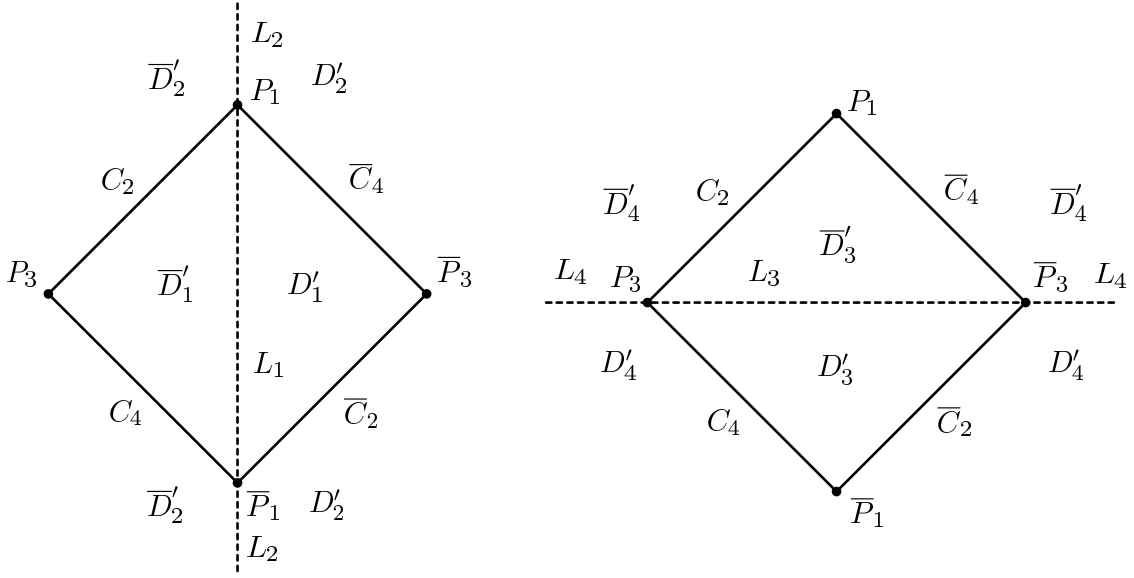}
\caption{The 8 cones meeting at singularities of $\tilde{Z}$.
The broken lines are the images of the four $G$-invariant twistor lines, 
which separate $D_j$ and $\ol{D}_j$ for $1\le j\le 4$. The 
rational curves $C_j, \overline{C}_j, j = 2, 4$ and $L_j, 1\le j\le 4$, 
are the intersection of the corresponding divisors.}
\label{fig-4faces}
\end{figure}

Hence any small resolution of $\tilde{Z}$ preserving the real 
structure falls into exactly one of the following:

\vspace{2mm}

\noindent
($\ast$) $\{D'_1,\ol D'_2\}$ and $\{ D'_3, \ol D'_4\}$ are 
blown-up pairs near $P_1$ and $P_3$, respectively, or 
$\{\ol D'_1,D'_2\}$ and $\{ \ol D'_3,  D'_4\}$ 
(the complementary pairs) are blown-up pairs near $P_1$ and $P_3$ 
respectively.
\vspace{2mm}

\noindent
$(\ast)'$ $\{D'_1,\ol D'_2\}$ and $\{\ol D'_3, D'_4\}$ are 
blown-up pairs near $P_1$ and $P_3$, respectively, or 
$\{\ol D'_1,D'_2\}$ and $\{ D'_3, \ol D'_4\}$ 
(the complementary pairs) are blown-up pairs near $P_1$ and $P_3$ 
respectively.

Obviously, each of these cases contain two ways of resolutions.
Consequently, for each case,  we obtain two (non-singular) 3-folds.
Next we see that these two spaces in each case are biholomorphic. 
For this, we define a new matrix $U_0$ by
\begin{align}
U_0:={\rm{diag}}(1,1,1,-1,1,1).
\end{align}
It is immediate to see (from \eqref{Q12}) that $U_0(\tilde Z)=\tilde Z$ holds.
We denote this involution on $\tilde Z$ by the same letter $U_0$.
Note that $U_0$ commutes with the real structure.

\begin{proposition}\label{prop:identify}
Let $\nu_1:Z_1\to \tilde Z$ and $\nu_2:Z_2\to \tilde Z$ be the two  
resolutions of $\tilde Z$ in the case $(\ast)$, and 
$\nu'_1:Z'_1\to \tilde Z$ and $\nu'_2:Z'_2\to \tilde Z$ be the two  
resolutions of $\tilde Z$ in the case $(\ast)'$.
Then the involution $U_0$ on $\tilde Z$ lifts as a biholomorphic map 
$Z_1\to Z_2$ and $Z'_1\to Z'_2$.
Furthermore, the last two biholomorphic maps commute with the real structure.
\end{proposition}

\begin{proof}
It suffices to verify that $U_0$ maps the blow-up pairs to the 
(complementary) blow-up pairs.
By elementary computations, we have $U_0(D'_1)=D'_2, U_0(D'_3)=D'_4$.
This immediately implies the former claim of the proposition.
The latter claim follows from the commutativity of $U_0$ with the real structure.
\end{proof}

By Proposition \ref{prop:identify}, we can identify $Z_1$ and $Z_2$, 
and also $Z'_1$ and $Z'_2$.
Next we show that the latter two spaces  are {\em not}\,  twistor spaces:

\begin{proposition}\label{prop:sr100}
Let $Z'_1$ and $Z'_2$ be as above and $\sigma_1'$ and $\sigma_2'$ the real 
structure induced by that on $\tilde Z$.
Then $(Z'_1,\sigma'_1)$ and $(Z'_2,\sigma'_2)$ are not twistor spaces.
\end{proposition}

\begin{proof}
By Proposition \ref{prop:identify},
it suffices to show the claim for $(Z'_1,\sigma'_1)$.
In $\CP^5$ we define  
\begin{align}
\begin{split}
C_2:=\{z_2=z_3=w_1=w_5=0\},\,\,
\ol C_2:=\{z_2=z_3=w_0=w_4=0\},\\
C_4:=\{z_2=z_3=w_0=w_5=0\},\,\,
\ol C_4:=\{z_2=z_3=w_1=w_4=0\}.\\
\end{split}
\end{align}
It is immediate to see that these are  $G$-invariant non-singular 
rational curves in $\tilde Z$.
Moreover, each of these 4 curves goes through exactly two singular 
points of $\tilde Z$ (see Figure \ref{fig-4faces}).
We further define
\begin{align}\label{L_j}
L_j:=D'_j\cap \ol D'_j,\,\,\,1\le j\le 4,
\end{align}
recalling from above that these are precisely the images of the $G$-invariant 
twistor lines. 
Suppose that $Z'_1$ is a twistor space.
Then by Lemma \ref{lemma-inv7}, these are the images of the  
four $G$-invariant twistor lines (under $\Psi$).
We use the same letters to mean the strict transforms into 
$Z'_1$ of these curves.
Further, let $C_1,\ol C_1,C_3,\ol C_3$ be the exceptional curves 
of the small resolution $Z'_1\to Z$.
Then in the small resolution $Z'_1$, the 8 curves 
$C_1,C_2,C_3,C_4,\ol C_1,\ol C_2,\ol C_3$ and $\ol C_4$ form an `octagon'.
(This is true for any small resolution of $\tilde Z$.)
Further, under the present choice of the small resolution, 
the curves $L_j$ in $\tilde Z'_1$ can be seen to be configured as in the 
left diagram in Figure \ref{fig:octagon}.

We make a short remark on how Figure \ref{fig:octagon} is 
obtained. For example, consider the the first small resolution in $(\ast)'$.
Then the blow-up pair at $P_1$ is $\{D'_1,\ol D'_2\}$.
This means that by the effect of the resolution, $L_1$ and 
$\ol C_4$ are separated by the exceptional curve $C_1$ since $D'_1$ 
is blown-up at $P_1\,(=L_1\cap \ol C_4)$.
At the same time, $C_2$ and $L_2$ are separated by $C_1$ since $\ol D'_2$ 
is blown-up at $P_1\,(=C_2\cap L_2$). As a result, near $C_1$ 
the situation becomes as in the left diagram in Figure \ref{fig:octagon}.
Similar reasoning applies to all other edges of the octagon. 

\begin{figure}
\includegraphics{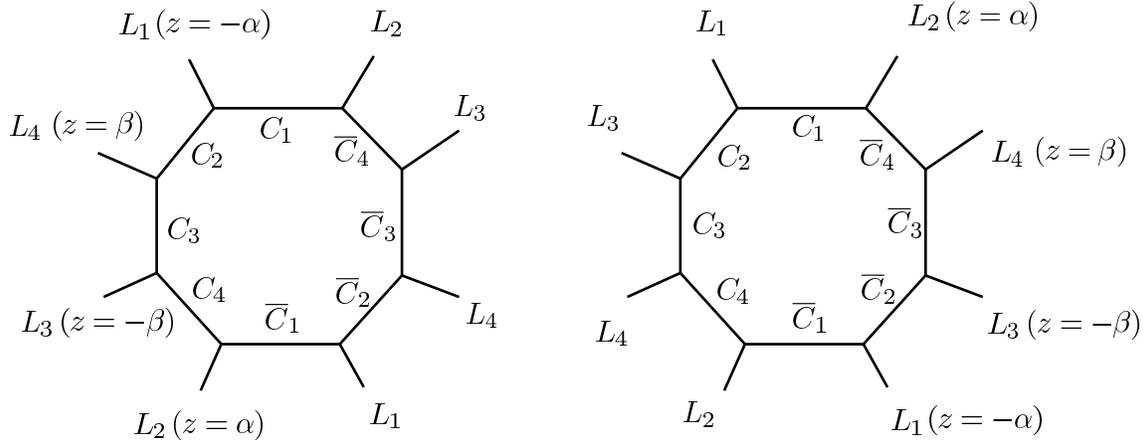}
\caption{Octagons formed by the 8 torus-invariant rational curves and the 
configuration of torus-invariant twistor lines. The left figure is 
for one of the two incorrect small resolutions, and the right
figure is for one of the two correct small resolutions.}
\label{fig:octagon}
\end{figure}

Next, we let $z:=z_3/z_2$ (where $z_2, z_3$ are part of the homogeneous 
coordinates of $\mathbb{CP}^5$) and consider it as a non-homogeneous 
coordinate on $\mathbb{CP}^1=\{(z_2,z_3)\}$.
Then by \eqref{rs1} the real structure on the last $\mathbb{CP}^1$ 
is given by $z\mapsto -\ol z$, so that the real locus is given by 
$\{z\in\mathbb C\set z\in i\mathbb R\}$.
Moreover by the definition of $L_j$, we have
\begin{align}
\begin{split}
z=-i\alpha {\text{\, on $L_1$}},\,\,
z=i\alpha {\text{\, on $L_2$}},\\
z=-i\beta {\text{\, on $L_3$}},\,\,
z=i\beta {\text{\, on $L_4$}}.\\
\end{split}
\end{align}
These mean that under the (meromorphic) $G$-quotient map 
$Z'\to \mathbb{CP}^1$ which is induced by the projection 
$(w_0,w_1,z_2,z_3,w_4,w_5)\mapsto (z_2,z_3)$, each $L_j$ 
is mapped to the point 
\begin{align}
\begin{split}
\label{4pts}
&z=-i\alpha \mbox{ for } j=1, \ z=i\alpha 
\mbox{ for } j=2, \\
&z=-i\beta \mbox{ for } j=3, \ z=i\beta \mbox{ for } j=4.
\end{split}
\end{align}
As Poon's metric is a special form of a Joyce metric,
we will next apply the theorem of Fujiki \cite[Theorem 9.1,~1)]{Fujiki2000},
which identifies the $(n+2)$ real parameters involved in the 
construction of Joyce metrics on $n\#\CP^2$ and the twistorial 
invariant that specifies the positions of the reducible members 
in the pencil $|F|^K$ (which in our case are 
$D_j+\ol D_j$, $1\le j\le 4$). Consequently, the four points in 
\eqref{4pts} can be canonically 
regarded as points on the boundary $\partial\mathcal H^2$ (where the 
Joyce metric is constructed 
on $K\times\mathcal H^2$). Furthermore, since the twistor fibration 
map $Z\to 2 \# \mathbb{CP}^2$ is $K$-equivariant, we have the diagram
\begin{equation}
\label{cd2}
 \CD
Z@<<<\left(
\cup_{i=1}^4C_i\right)
\cup 
\left(\cup_{i=1}^4\ol C_i\right) \\
@VVV @VV{/\langle \sigma_1'\rangle}V \\
2\#\CP^2@<<<
4\,\, K{\text{-invariant 2-spheres}}
\\
@VV{/K}V @VV{/K}V \\
\mathcal H^2\cup\partial \mathcal H^2
@<<< \partial \mathcal H^2
 \endCD
 \end{equation}
where all horizontal arrows mean the obvious inclusions as  subsets.
In particular, we have an isomorphism
\begin{align}
\left(
(\cup_{i=1}^4C_i)
\cup 
(\cup_{i=1}^4\ol C_i)
\right) / \langle K,\sigma'_1\rangle
\simeq
\partial\mathcal H^2,
\end{align}
where $\langle K,\sigma'_1\rangle$ means the automorphism group of 
$Z'$ generated by $K$ and $\sigma'_1$.
Therefore, looking the left diagram of Figure \ref{fig:octagon}, we see that the image of 
the four $K$-fixed points of the $K$-action on $2\mathbb{CP}^2$ under the quotient map
\begin{align*}
2\#\CP^2\to (2\#\CP^2)/K \simeq \mathcal H^2\cup\partial\mathcal H^2 
\end{align*}
are configured along $\partial \mathcal H^2$ in the order 
\begin{align}
-\alpha,\,\,\beta,\,\,-\beta,\,\,\alpha.
\end{align}
But as $\alpha>0$ and $\beta>0$, the 4 numbers cannot be configured 
in this order, even up to cyclic permutation and reversing the 
orientation.
Therefore, the $L_j$-s cannot be configured as in the left diagram in 
Figure \ref{fig:octagon}.
This means that the small resolutions in $(\ast)'$ are 
not the twistor space, as claimed.
\end{proof}

Thus we have obtained the small resolutions of the projective variety 
$\tilde Z$ which give the twistor space in completely explicit form.
Namely, such  small resolutions are exactly the two ones in $(\ast)$.
We remark that for the former among the two correct small resolutions, 
the torus-invariant twistor lines are configured as in the right diagram 
in Figure \ref{fig:octagon}; the latter case becomes the mirror 
image of this.

\subsection{Determination of the conformal isometry group (for $2\#\CP^2$)}
\label{detconf}

In this subsection we determine,
among the automorphisms in Lemma \ref{lemma:autproj} (parameterized by 16 tori),
which automorphisms actually lift to the twistor space. 
(Note that in general automorphisms of the base do not 
necessarily lift to a small resolution.)
We begin with Case (I).
\begin{proposition}
\label{prop-key}
 Let $U$ be the $6\times6$ matrix of the form
\begin{eqnarray}\label{U1'}
U=
\left(
\begin{array}{ccc}
A_{11}&O &O \\
O&A_{22} &O \\
O& O& A_{33}\\
\end{array} 
\right),
\end{eqnarray}
where
\begin{equation}
\label{ent33'}
A_{22}=
\left(
\begin{array}{cc}
     1&0   \\
     0&  1 \\
\end{array}
\right)
\,\,\,{\text{ or }}\,\,\,
\left(
\begin{array}{cc}
     1&0   \\
     0&  -1 \\
\end{array}
\right),
\end{equation}
 and
\begin{align}
\label{ent31''}
A_{11}=\left( 
\begin{array}{cc}
a &0 \\
0 &\ol a \\
\end{array} 
\right){\text{ or }}
\left( 
\begin{array}{cc}
0 &a \\
\ol a &0 \\
\end{array} 
\right),\,\,
A_{33}=\left( 
\begin{array}{cc}
b &0 \\
0 &\ol b \\
\end{array} 
\right){\text{ or }}
\left( 
\begin{array}{cc}
0 &b \\
\ol b &0 \\
\end{array} 
\right),\,\,|a|=|b|=1.
\end{align}
(These are necessary conditions obtained in Lemmas \ref{lemma-zero} 
and \ref{lemma:autproj}.)
Then $U$ lifts to the twistor space if and only if  
$A_{22}$, $A_{11}$ and $A_{33}$ take the following combinations:
\begin{itemize}
\item $A_{22}=I_2$ and $A_{11}, A_{33}$ are diagonal,
\item $A_{22}=I_2$ and $A_{11}, A_{33}$ are off-diagonal,
\item  $A_{22}={\rm{diag}}(1,-1)$, $A_{11}$ is diagonal and $ A_{33}$ is off-diagonal, 
\item  $A_{22}={\rm{diag}}(1,-1)$, $A_{11}$ is off-diagonal and $ A_{33}$ is diagonal.
\end{itemize}
\end{proposition}
\begin{remark}{\em
This proposition means that the natural injective homomorphism 
\begin{equation}
{\rm{Aut}}^{\sigma}Z\to{\rm{Aut}}^{\sigma}\tilde{Z}
\end{equation}
is not surjective.
Namely, even if we restrict to the real resolutions, the projective models 
can have strictly larger symmetries than that of the twistor space.
}
\end{remark}
\begin{proof}[Proof (of Proposition \,\ref{prop-key})]
We determine whether the projective transformation $U$ lifts to a small 
resolution, by using the obvious fact that   an automorphism $U$ of $\tilde Z$ lifts to a small resolution $Z$ if and only if 
$U$ maps blow-up pairs at any ordinary nodes of $\tilde Z$  (in the sense of Section \ref{ss:detsr}; see $(\ast)$)  to a blow-up pair.
More concretely:

1)
If $U$ fixes $P_j$ ($j=1$ or $3$), then $U$ can be lifted to a small 
resolution of $\tilde{Z}$ at $P_j$ if and only if $U$ preserves {\em each} pair of
divisors. (If $j=1$, this means $\{U(D_1),U(\ol D_2)\}=\{D_1,\ol D_2\}$;
if $\{U(D_1),U(\ol D_2)\}=\{\ol D_1, D_2\}$, $U$ does not lift on 
any small resolutions.
If $j=3$, this means $\{U(D_3),U(\ol D_4)\}=\{D_3,\ol D_4\}$;
if $\{U(D_3),U(\ol D_4)\}=\{\ol D_3, D_4\}$, $U$ does not lift on 
any small resolutions.)
In these cases, $U$ can also be lifted to any small resolution 
(of $P_j$) automatically.

2)
If $U(P_1)=\ol P_1$, then $U$ can be lifted to small resolutions of 
$\tilde{Z}$ at $P_1$ and $\ol P_1$ which preserve the real structure if and only if 
$\{U(D_1),U(\ol D_2)\}=\{\ol D_1,D_2\}$.
Similarly, if $U(P_3)=\ol P_3$, $ U$ can be lifted to small 
resolutions at $P_3$ and $\ol P_3$ which preserves the real 
structure if and only if $\{U(D_3),U(\ol D_4)\}=\{\ol D_3,D_4\}$.

\vspace{1mm}
First we examine $U$ of \eqref{U1'} in the case where
 $A_{22}=I_2$ and $A_{11}, A_{33}$ are diagonal.
These $U$ fix all four singularities of $\tilde{Z}$ and 
leave any $D_j$ and $\ol D_j$ ($1\le j\le 4$) invariant. 
Hence by 1) above, we conclude that such $U$ lift to any small resolution of $\tilde{Z}$.
In particular, $U$ lifts to an automorphism of the twistor space $Z$.
Since these $U$ include the identity matrix, they form the identity component 
of the automorphism group. 

Next, if $A_{22}={\rm{diag}}(1,-1)$, and $A_{11}, A_{33}$ are diagonal,
then $U(P_1)=P_1$ and $ U(D_1)=D_2$ hold.
Hence by 1), these $U$ do not lift to any small resolution.
If
$A_{22}=I_2$ and $A_{11}$ is diagonal, and $A_{33}$ is off-diagonal, 
then
$U(P_1)=P_1$ and $U(D_1)=\ol D_1$ hold. 
Hence by 1), these $U$ do not lift to any small resolution.
If  $A_{22}={\rm{diag}}(1,-1)$, and $A_{11}$ is diagonal and $A_{33}$ 
is off-diagonal, then
$U(P_1)=P_1$ and $U(D_1)=\ol D_2$.
Hence by 1), these $U$ lift to any small resolution at $P_1$ and $\ol P_1$.
Further since $U(P_3)=\ol P_3$ and $U(D_3)=D_4$, by 2) this time, 
we conclude that these $U$ lift to any small resolution 
at $P_3$ and $\ol P_3$ as long as they preserve the real structure.
Hence these $U$ lift to an automorphism of the twistor space $Z$.
If $A_{22}=I_2$, $A_{11}$ is off-diagonal and $A_{33}$ is diagonal, 
then we have $U(P_3)=P_3$ and $U(D_3)=\ol D_3$. Hence by 1), 
these $U$ do not lift to $Z$.
If $A_{22}={\rm{diag}}(1,-1)$, $A_{11}$ is off-diagonal and $A_{33}$ is diagonal,
then we have  $U(P_1)=\ol P_1$, $U(D_1)=D_2$, $U(P_3)=P_3$ and $U(D_3)=\ol D_4$.
Hence by 2) and 1), these $U$ do lift to the twistor space $Z$.
If $A_{22}=I_2$ and $A_{11}$ and $A_{33}$ are off-diagonal, 
then we have $U(P_1)=\ol P_1$, $U(D_1)=\ol D_1$, $U(P_3)=\ol P_3$, 
and $U(D_3)=\ol D_3$. Hence by 2), these $U$ do lift to the twistor space $Z$.
Finally, if $A_{22}={\rm{diag}}(1,-1)$ and $A_{11}$ and $A_{33}$ are off-diagonal, 
then we have $U(P_1)=\ol P_1$, $U(D_1)=\ol D_2$. 
Hence by 2), these $U$ do not lift to $Z$.
This completes the proof of Proposition \ref{prop-key}.
\end{proof}

Next we consider Case (II).
In order to simplify notation, we put
\begin{align}
A_{22}^+=\i
\begin{pmatrix}
0 & 1\\
\alpha\beta& 0\\
\end{pmatrix},\
A_{22}^-=\i
\begin{pmatrix}
0 & -1\\
\alpha\beta& 0\\
\end{pmatrix}.
\end{align}
\begin{proposition}
\label{prop-key2}
 Let $U$ be a \,$6\times6$ matrix of the form
\begin{eqnarray}\label{U2'}
U=
\left(
\begin{array}{ccc}
O&O &A_{13} \\
O&A_{22} &O \\
A_{31}& O& O\\
\end{array} 
\right),
\end{eqnarray}
where
$A_{22}=A_{22}^+$ or $A_{22}=A_{22}^-$ 
 and
\begin{align}
\label{ent31'''}
A_{13}=\left( 
\begin{array}{cc}
a &0 \\
0 &-\ol a \\
\end{array} 
\right){\text{ or }}
\left( 
\begin{array}{cc}
0 &a \\
-\ol a &0 \\
\end{array} 
\right),\,
A_{31}=\left( 
\begin{array}{cc}
b &0 \\
0 &-\ol b \\
\end{array} 
\right){\text{ or }}
\left( 
\begin{array}{cc}
0 &b \\
-\ol b &0 \\
\end{array} 
\right),
\begin{array}{l}
|a|=\beta \\
|b|=\alpha.  \\
\end{array} 
\end{align}
(These are necessary conditions obtained in Lemmas 
\ref{lemma-zero} and \ref{lemma:autproj}.)
Then $U$ lifts to the twistor space if and only if  
$A_{22}$, $A_{13}$ and $A_{31}$ take the following combinations:
\begin{itemize}
\item $A_{22}=A_{22}^-$ and $A_{13}, A_{31}$ are diagonal,
\item $A_{22}=A_{22}^-$ and $A_{13}, A_{31}$ are off-diagonal,
\item  $A_{22}=A_{22}^+$, $A_{13}$ is diagonal and $ A_{31}$ is off-diagonal, 
\item  $A_{22}=A_{22}^+$, $A_{13}$ is off-diagonal and $ A_{31}$ 
is diagonal.
\end{itemize}
\end{proposition}
\begin{proof}
We define a matrix of the form \eqref{U2'} (satisfying \eqref{ent31'''}) by 
\begin{align}\label{iota}\Lambda:=
\begin{pmatrix}
 0&0 &0&0&\beta&0\\
0&0&0&0&0&-\beta\\
0&0&0&-\i&0&0\\
0&0&\i\alpha\beta&0&0&0\\
\alpha&0&0&0&0&0\\
0&-\alpha&0&0&0&0
\end{pmatrix}.
\end{align}
We note $\Lambda^2=\alpha\beta I$, so that 
$\Lambda$ defines an involution of $\tilde{Z}$.
Moreover, we have
\begin{align}
\Lambda(D'_1)= D'_3,\,\Lambda(\ol D'_2)= \ol D'_4,\,
\Lambda(D'_3)=D'_1,\,\Lambda( \ol D'_4)=\ol D'_2.
\end{align}
In particular we have $\{\Lambda(D'_1),\Lambda(\ol D'_2)\}=\{ D'_3, \ol D'_4\}$ and 
$\{\Lambda(D'_3),\Lambda(\ol  D'_4)\}=\{D'_1,\ol D'_2\}$
Noting that $ \Lambda(P_1)=P_3,\,\Lambda(P_3)=P_1$,
this means that $\Lambda$ maps any blow-up pairs to blow-up pairs
for the small resolutions in the case $(\ast)$.
Therefore $\Lambda$ lifts to $Z$ if (and only if) the above condition 
$(\ast)$ is  satisfied.
Hence $\Lambda$ lifts to the twistor space $Z$.

Having done this, 
for any matrix $U$ of the form \eqref{U2'} (subject to \eqref{ent31'''})
we consider the product $\Lambda U$.
If $A_{13}$ and $A_{31}$ (in the matrix $U$) are diagonal and $A_{22}=A_{22}^-$, 
up to a non-zero constant,
the product  $\Lambda U$ becomes of the first form 
in Proposition \ref{prop-key}.
Hence by the proposition $\Lambda U$ lifts to $Z$.
Therefore, as $\Lambda$ lifts to $Z$ for the small resolutions in $(\ast)$ 
as above, we obtain that 
these $U$ lifts to $Z$
for the small resolutions in the case $(\ast)$.
Similarly, if $A_{13}$ and $A_{31}$ (in the matrix $U$) are off-diagonal 
and $A_{22}=A_{22}^-$, then  up to a non-zero constant, 
the product  $\Lambda U$ becomes of the second form, 
so that $U$ lifts to $Z$ for the small resolutions in $(\ast)$.
 If  $A_{13}$ and $A_{31}$ are diagonal and off-diagonal respectively 
and $A_{22}=A_{22}^+$,
then up to  a non-zero constant,  $\Lambda U$ becomes of the forth form, 
so that $U$ lifts to $Z$
for the small resolutions in  $(\ast)$.
If $A_{13}$ and $A_{31}$ are off-diagonal and diagonal 
respectively and $A_{22}=A_{22}^+$, then up to a non-zero constant,  $\Lambda U$ becomes of the 
third form, 
so that $U$ lifts to $Z$ for the small resolutions in  $(\ast)$.
Further, it can be readily checked that if $U$ is not of these 4 forms, 
then $\Lambda U$ does not coincide with any of the 4 forms 
and therefore $U$ does not lift to $Z$ for the small resolutions 
in  $(\ast)$ by Proposition \ref{prop-key}.
Thus we have proved the claim of the proposition.
\end{proof}

By Propositions \ref{prop-key}, \ref{prop-key2} and \ref{prop:sr100}, we have 
obtained explicit representations of all conformal isometries of 
Poon's metrics on $2\#\CP^2$ by $6\times 6$ matrices.
Namely, we have obtained the image of the (injective) homomorphism 
\eqref{hom3} explicitly.
\section{Geometric interpretation}
\label{geometric}
In this subsection, we investigate the geometry of the 
conformal automorphisms obtained in the previous sections.
We begin with the following
\begin{lemma}
\label{lemma:sf}
Let $n\ge 2$ and $[g_{\rm{\,LB}}]$ be a LeBrun metric on 
$n\#\CP^2$. 
Then (i) if $n\ge 3$, there exists a unique {\rm{U(1)}}-subgroup 
of ${\rm{Aut}}(g_{\,\rm{LB}})$ 
which acts semi-freely on $n\#\CP^2$,
(ii) if $n=2$, the number of such ${\rm{U(1)}}$-subgroups is two.
\end{lemma}
\begin{proof}
Let $p_1,\dots,p_n\in\H^3$ be the monopole points of $[g_{\rm{\,LB}}]$.
Then the structure group \U(1) of the principal bundle 
over $\H^3 \setminus \{p_1,\dots,p_n\}$ acts semi-freely 
on $n\CP^2$, and it coincides with the identity component 
of ${\rm{Aut}}(g_{\,\rm{LB}})$ if and only if the $n$ points do
not lie on a common geodesic. 
Therefore to prove (i) it suffices to consider the case 
that $p_1,\dots,p_n$ are contained on a common geodesic.
If the last condition is satisfied, the identity component 
of ${\rm{Aut}}(g_{\,\rm{LB}})$ becomes the torus $K$.
Note that for $n=2$, this condition is automatically satisfied.

The $K$-action on $n\#\CP^2$ is obtained as follows.
First consider a standard $K$-action on $\CC^2$, 
which is given by $(z,w)\mapsto (sz,tw)$ for 
$(s,t)\in \U(1) \times \U(1)$. 
We blow-up $\CC^2$ at $n$ points in such a way that the blown-up 
points are always on the unique $K$-fixed point of the 
strict transform of the $z$-axis.
Let $\tilde{\CC}^2$ be the resulting complex (toric) surface.
Next, we add a  point at infinity to $\tilde{\CC}^2$.
Then by reversing the standard orientation, we obtain  
$n\#\CP^2$ with a $K$-action. 
(Over the open subset 
$\tilde{\CC}^2\subset n\#\CP^2$, $[g_{\rm{\,LB}}]$ 
contains a K\"ahler scalar-flat metric with a $K$-action.)
As this $K$-action contains a $\U(1)$-subgroup acting semi-freely 
(which is explicitly given by $\{(s,t)\set s=1\}$), it can be 
identified with the identity component of ${\rm{Aut}}(g_{\,\rm{LB}})$ 
(in the present situation).
Hence to prove the lemma it is enough to classify all 
$\U(1)$-subgroups of $K$ which act semi-freely on $\tilde {\CC}^2$.

If $K_1\subset K$ is such a $\U(1)$-subgroup, $K_1$ has non-isolated 
fixed points \cite[Proposition 1]{LeBrun1993}.
Hence, since the $K$-action on $\tilde{\CC}^2$ is free on the 
preimage of $\CC^2 \setminus \{zw=0\}$, the subgroup $K_1$ 
has to fix the strict transform of the $z$-axis or the $w$-axis, 
or some exceptional curve of the blow-up $\tilde{\CC}^2\to\CC^2$.
On these $K$-invariant subsets, the $K$-action is explicitly given 
by multiplication by
\begin{align}
t,\,s^{-1},\, ts^{-1},\,ts^{-2},\dots,ts^{-n},
\end{align}
respectively.
Namely, all subgroups having non-isolated fixed locus are explicitly given by
\begin{align}
\{t=1\},\,\{s=1\},{\text{ and }} \{t=s^k\}\,(1\le k\le n)\,.
\end{align}
Since $n\ge 2$ the first one acts non-semi-freely, whereas the 
second one acts semi-freely.

For the remaining subgroups $\{t=s^k\}$ ($1\le k\le n$), 
the action on the $(n+2)$ $K$-invariant subsets (in the last paragraph) 
\begin{align}
s^k,\,s^{k-1},\,s^{k-2},\,\dots,s^{k-n},\,s^{-1}.
\end{align}
Hence the action becomes semi-free if and only if $n=2$ and $k=1$.
This means that if $n\ge 3$ the subgroup $\{s=1\}$ is the 
unique $\U(1)$-subgroup acting semi-freely, 
and if $n=2$, the subgroups $\{s=1\}$ and $\{t=s\}$ are all such 
subgroups. Thus we have obtained the claim of the lemma.
\end{proof}
We return to the case of $2\#\CP^2$.
Recall that in the proof of Lemma \ref{lemma-inv7} we have 
defined two $\CC^*$-subgroups $G_1$ and $G_3$ 
(explicitly defined as \eqref{G1} and \eqref{G3}).
\begin{lemma}\label{lemma:sf2}
Viewing the group $G=\CC^*\times\CC^*$ (acting on Poon's
twistor space) as the complexification of  
$K={\rm{U(1)}}\times{\rm{U(1)}}$ (acting on Poon's metric),
the subgroups $G_1$ and $G_3$ of $G$ are exactly the complexification 
of the two {\rm{U(1)}}-subgroups acting semi-freely on $2\#\CP^2$.
\end{lemma}
\begin{proof}
We freely use notations in the previous section.
It suffices to show that $G_1$ and $G_3$ act semi-freely on the 
twistor space $Z$.
By their explicit forms \eqref{G1} and \eqref{G3}, $G_1$ and 
$G_3$ clearly act semi-freely on   $\CP^5$. Therefore 
they act semi-freely on the projective model $\tilde{Z}$.
Hence it is enough to show that they act semi-freely on the 
exceptional curves  $C_1,C_3,\ol C_1$ and $\ol C_3$ of the small resolutions $Z\to \tilde{Z}$. 
The weights for the $G_1$ and $G_3$-actions on these curves can readily computed
by using the $G$-invariant divisors $D'_i$ and $\ol D'_i$ $(1\le i\le 4)$, and 
they become either $0$ or $1$.
Thus we conclude that $G_1$ and $G_3$ act semi-freely on $Z$.
\end{proof}
Let $K_1$ and $K_3$ be the $\U(1)$-subgroups of $K$ whose 
complexifications are $G_1$ and $G_3$, respectively.
We know that these are all of the $\U(1)$-subgroups 
acting semi-freely. For these subgroups, we set 
\begin{align}
X_0=\{p\in 2\CP^2\set 
{\text{the isotropy subgroup of $K_1$ at $p$ is \{{\rm{Id}}}}\}\},
\end{align}
and 
\begin{align}
Y_0=\{p\in 2\CP^2\set 
{\text{the isotropy subgroup of $K_3$ at $p$ is \{{\rm{Id}}}}\}\}.
\end{align}
From the  proof of Lemma \ref{lemma:sf} we know $X_0\neq Y_0$.
Let $p_1$ and $p_2$ be the image of the two isolated fixed points 
of the $K_1$-action  under the quotient map 
$2\#\CP^2\to2\#\CP^2/K_1$.
Similarly, let $q_1$ and $q_2$ be the image of the two 
isolated fixed points of the $K_3$-action  under the 
quotient map $2\#\CP^2\to2\#\CP^2/K_3$.
Then since $[g_{\rm{\,LB}}]$ is $K_1$-invariant, 
by the result of LeBrun \cite{LeBrun1993}, the quotient space 
$\H^3_1:=(X_0/K_1)\cup\{p_1,p_2\}$ becomes a 
$3$-manifold  equipped with a hyperbolic metric and 
$g_{\rm{\,LB}}$ is obtained by the hyperbolic ansatz with
monopole points $p_1$ and $p_2$.
Similarly, $\H^3_3:=(Y_0/K_3)\cup\{q_1,q_2\}$ 
becomes a $3$-manifold  equipped with a hyperbolic metric 
and $g_{\rm{\,LB}}$ is obtained by the hyperbolic ansatz 
whose monopole points are $q_1$ and $q_2$.
Thus any Poon metric on $2 \# \CP^2$ has the following 
double fibration:
\begin{align}\label{df}
 \xymatrix{ & 2\#\CP^2 \ar[dl]_{\pi_1} \ar[dr]^{\pi_3} & \\ 
\H^3_1\cup\partial \H^3_1&& \H^3_3\cup\partial \H^3_3. }
\end{align}
Here, $\pi_1$ and $\pi_3$ are the quotient maps by the 
$K_1$-action and $K_3$-action, respectively, and 
$\partial \H^3_1(\simeq S^2)$ and 
$\partial \H^3_3(\simeq S^2)$ are the images of the
non-isolated fixed locus of the $K_1$-action and 
$K_3$-action, respectively. Note that if $n\ge 3$, an analogous double fibration 
does not exist by Lemma \ref{lemma:sf}. 

By Propositions \ref{prop-key} and \ref{prop-key2}, when 
$n=2$ the group ${\rm{Aut}}(g_{\rm{\,LB}})$ consists of 8 tori.

\begin{definition}\label{def:H}{\em
We define $H$ to be a subgroup of the full conformal isometry 
group ${\rm{Aut}}(g_{\rm{\,LB}})$ consisting of the 4 tori in 
Proposition \ref{prop-key}; namely $H$ consists of automorphisms  
which are lifts of automorphisms of the projective model $\tilde{Z}$ 
represented by matrices of `diagonal type'.}
\end{definition}

\begin{proposition}\label{prop-key3}
  The image of the subgroup $H$ under the homomorphism 
\begin{align*}
{\rm{Aut}}(g_{\LB})\lra {\rm{GL}}(H^0(Z,F))
\end{align*} 
in \eqref{hom2} preserves the two subspaces 
$H^0(Z, F)^{G_1}$ and  $H^0(Z, F)^{G_3}$. 
\end{proposition}
\begin{proof}
Take any $\Phi\in H$ and let $U\in {\rm{GL}}(6,\CC)$ 
be the image of $H$ under the homomorphism, where we are 
using $\{w_0,w_1,z_2,z_3,w_4,w_5\}$ as a basis of 
$H^0(Z,F)\simeq\CC^6$ as before.
 By the definition of the subgroup $H$, 
$U$ must be of the form
\begin{eqnarray}\label{U7}
\left(
\begin{array}{ccc}
A_{11}&O &O \\
O&A_{22} &O \\
O& O&A_{33}\\
\end{array} 
\right),\hspace{2mm}
A_{11},A_{22},A_{33}\in {\rm{GL}}(2,\CC).
\end{eqnarray}
On the other hand, by \eqref{G-action1}, the two 
subspaces are explicitly given by
\begin{align}
H^0(Z, F)^{G_1}=\langle z_2,z_3,w_4,w_5\rangle, \mbox{ and }
H^0(Z, F)^{G_3}=\langle w_0,w_1,z_2,z_3\rangle.
\end{align}
These directly imply the claim of the proposition.
\end{proof}
\begin{proposition}
Let $n=2$ and 
 $H\subset {\rm{Aut}}(g_{\,\rm LB})$ be as in Definition \ref{def:H}.
Then there are  homomorphisms 
\begin{align}
\rho_1:H\lras
{\rm{Aut}}(\H^3_1;p_1,p_2)
\end{align}
and
\begin{align}
\rho_3:H\lras
{\rm{Aut}}(\H^3_3;q_1,q_2)
\end{align}
such that $\rho_j (\Phi) = \phi$, where $\Phi$ is any 
lift of $\phi$ obtained in Proposition \ref{hypconf}. 
\end{proposition}
\begin{proof}
We recall that we have defined the linear projections 
$f_j:\CP^5\to\CP^3$ for $j=1,3$ which are 
explicitly given by \eqref{proj1}-\eqref{proj2}.
By the definition and \eqref{G-action1}, the composition 
$f_j\circ\Psi:Z\to\CP^3$ is exactly the rational 
map associated to the vector space $H^0(Z, F)^{G_j}$.
The image $f_j\circ\Psi(Z)=f_j(\tilde{Z})$ (explicitly given as 
\eqref{image1}) is isomorphic to $\CP^1\times\CP^1$, 
on which $K_j$ acts trivially.
Moreover, by Proposition \ref{prop-key3}, $H$ automatically 
preserves the quadric $f_j(\tilde{Z})$ for $j=1,3$.
(This is also clear from Proposition \ref{prop-key} and \eqref{image1}.)
Hence for $j=1,3$ we obtain two homomorphisms
\begin{equation}\label{hom100}
H\lra {\rm{Aut}}(\CP^1\times\CP^1).
\end{equation}
Furthermore, as we have considered those matrices $U$ which commute 
with the real structure, the image of these homomorphisms 
commutes with the natural real structure on $\CP^1\times\CP^1$.
Moreover, if $U$  is a matrix representing an element of $H$,
 we have $\{U(P_1), U(\ol P_1)\}=\{P_1,\ol P_1\}$ and  
$\{U(P_3), U(\ol P_3)\}=\{P_3,\ol P_3\}$.
If $C_1$ and $C_3$ respectively denote the exceptional curves 
(for $\Psi$) over the singular points $P_1$ and $P_3$ of $\tilde{Z}$ 
as before, by the twistor fibration $C_1$ and $C_3$ are 
mapped to the 2-spheres which are fixed by $K_1$-action 
and $K_3$-action, respectively.
Hence any $\Phi\in H$ leaves the boundary sphere 
$\partial \H^3_j\subset2\#\CP^2$ invariant 
for $j=1$ and $3$. Therefore, viewing $\CP^1\times\CP^1$  
as the minitwistor space of the hyperbolic space $\H^3_j$ 
as in the case $n\ge 3$, we obtain a homomorphism 
\begin{equation}\label{hom101}
\rho_j:H\lra {\rm{Aut}}(\H^3_j)\ (j=1,3).
\end{equation}
Moreover, the image of \eqref{hom100} preserves the set of 
discriminant curves $\{\mathcal C_1,\mathcal C_2\}$ of the 
map $f_j\circ\Psi$ by the same reason for the case $n\ge 3$ 
given in Proposition \ref{prop-LB10} (iii). 
Hence the image of \eqref{hom101} is contained in 
${\rm{Aut}}(\H^3_1;\,p_1,p_2)$ for $j=1$ and 
${\rm{Aut}}(\H^3_3;\,q_1,q_2)$ for $j=3$.
Furthermore, the homomorphism $\rho_j$ is an inverse 
of the lift in Proposition \ref{hypconf} by the same reason for 
the case $n \geq 3$ given in the final part of the proof of 
Proposition \ref{prop-hom}. This finishes the proof. 
\end{proof}
This means that the action of the subgroup $H$ preserves each 
of the two fibrations in \eqref{df} respectively.
On the other hand, for automorphisms not belonging to $H$, 
we have the following
\begin{proposition}
\label{bcprop}
If $\Phi\in {\rm{Aut}}(g_{\,\rm LB})$ satisfies $\Phi\not\in H$, 
$\Phi$ maps any fiber of $\pi_1$ to a fiber of $\pi_3$, and any fiber of 
$\pi_3$ to a fiber of $\pi_1$, where $\pi_1$ and $\pi_3$ are the 
quotient maps by the $K_1$-action and the $K_3$-action, respectively, as before.
\end{proposition}
\begin{proof}
Since the lift of the $K_j$-actions ($j=1,3$) on $2\#\CP^2$ 
to the twistor space is given by the restriction of the 
$G_j$-action to the real forms by Lemma \ref{lemma:sf2},  
it suffices to show that by any $\Phi\not\in H$, $G_1$-orbits 
are mapped to $G_3$-orbit, and $G_3$-orbits are mapped to $G_1$-orbits.
Let $U$ be a $6\times 6$ matrix corresponding to $\Phi\not\in H$.
Then $U$ is as in Proposition \ref{prop-key2}.
As $U$ contains 2 parameters $a$ and $b$ (satisfying $|a|=\beta$ 
and $|b|=\alpha$), we write $U=U(a,b)$ (to simplify notation). 
On the other hand, the subgroups $G_1$ and $G_3$ are explicitly 
given in \eqref{G1} and \eqref{G3}.
Let $B(s):={\rm{diag}}(s,s^{-1},1,1,1,1)$
and $C(t):={\rm{diag}}(1,1,1,1,t,t^{-1})$.
Then as $6\times 6$ matrices, we have the following relations
\begin{align}
\begin{split}
B(s)U(a,b)&=U(sa,b),\  U(a,b)B(s)=U(a,s^{-1}b),\\
C(t)U(a,b)&=U(a,tb),\  U(a,b)C(t)=U(t^{-1}a,b).
\end{split}
\end{align}
These imply that $U(a,b)$ interchanges $G_1$-orbits and 
$G_3$-orbits, as required.
\end{proof}
As an immediate consequence of the above discussion, we obtain the 
following
\begin{corollary}
Let $d_1$ and $d_3$ be the hyperbolic distance between $p_1$ and 
$p_2\in \H^3_1$, and $q_1$ and $q_2\in \H^3_3$, 
respectively. Then $d_1=d_3$ holds.
\end{corollary}

\subsection{Generators of the automorphism group}
Finally, we give generators of the full automorphism group 
${\rm{Aut}}(g_{\,\rm LB})$ in the case $n=2$.
(For $n\ge 3$ generators of ${\rm{Aut}}(g_{\,\rm LB})$ 
were already given in Theorem \ref{summarytheorem}).

\begin{proposition}
\label{Uprop}
Suppose $n=2$ and let $H\subset {\rm{Aut}}(g_{\,\rm LB})$ be as in Definition \ref{def:H}, 
and let ${\rm{Aut}}_0(g_{\,\rm LB})\,(\simeq K)$ be the identity 
component of ${\rm{Aut}}(g_{\,\rm LB})$. Then we have: 
(i) The subgroup $H$ is generated by ${\rm{Aut}}_0(g_{\,\rm LB})$ 
and two involutions.
(ii) ${\rm{Aut}}(g_{\,\rm LB})$ is generated by $H$ and an 
involution $\Lambda$ not belonging to $H$.
\end{proposition}
\begin{proof}
This is easy since we have explicit representation of 
${\rm{Aut}}(g_{\,\rm LB})$ as $6\times6$ matrices.
For (i), as the two involutions in $H$ we choose the ones
 represented by the following matrices
\begin{align}\label{iota1and2}\Lambda_1:=
\begin{pmatrix}
1 & &&&&\\
&1&&&&\\
&&1&&&\\
&&&-1&&\\
&&&&0&1\\
&&&&1&0
\end{pmatrix}
{\text{ and }}\ 
\Lambda_2:=
\begin{pmatrix}
0 &1 &&&&\\
1&0&&&&\\
&&1&&&\\
&&&-1&&\\
&&&&1&\\
&&&&&1
\end{pmatrix},
\end{align}
where a blank entry means $0$.
It is readily seen that $\Lambda_1^2=\Lambda_2^2=I$, 
$\Lambda_1$ and $\Lambda_2$ belong to mutually different non-identity connected 
components of $H$, and that the product $\Lambda_1\Lambda_2$ belongs to the 
remaining connected component of $H$. This means that the identity 
component and $\Lambda_1$ and $\Lambda_2$ generate the subgroup $H$.
Hence we obtain (i). Note that these corresponds to the transformations 
described in Theorem \ref{summarytheorem}. 

For (ii) we choose the involution $\Lambda$ given in \eqref{iota}.
As in the proof of Propositions \ref{prop-key}, $\Lambda$ defines an involution on 
the twistor space $Z$.
Since $\Lambda$ is off-diagonal type, we have $\Lambda\not\in H$. 
Further it is elementary by using Propositions \ref{prop-key} and
\ref{prop-key2} to show that for any one of the other $3$ components of 
${\rm{Aut}}(g_{\,\rm LB}) \setminus H$, we can find an element 
$U\in H$ for which the product $U\cdot \Lambda$ belongs to that component.
This means that $H$ and $\Lambda$ generate ${\rm{Aut}}(g_{\,\rm LB})$.
\end{proof}
The following proposition completes the proof of 
Theorem \ref{summarytheorem} above. 
\begin{proposition}
\label{D4prop}
As before, let ${\rm{Aut}}_0$ 
be the identity component  of \,${\rm{Aut}}_0(g_{\rm{LB}})$, 
which is obviously a normal subgroup of $H$.
Then the quotient group $H/ {\rm{Aut}}_0$ 
is isomorphic to $\mathbb Z_2\times\mathbb Z_2$. 
Moreover, the 
quotient ${\rm{Aut}}/ {\rm{Aut}}_0$ is isomorphic to 
${\rm{D}}_4$ (the dihedral group of order $8$).
\end{proposition}
\begin{proof}
The former claim readily follows from the explicit form of the 
matrices $U$ in Proposition \ref{prop-key}.
For the second claim, we first note that the group is non-Abelian, 
by the explicit form of the matrices $U$ in Proposition \ref{prop-key2}.
Therefore, it is isomorphic to either the quaternion group 
(the subgroup generated by $i$ and $j$ in the quaternions), 
or the dihedral group ${\rm{D}}_4$. But the former group cannot contain 
a subgroup which is isomorphic to $\mathbb Z_2\times\mathbb Z_2$. 
Therefore ${\rm{Aut}}/ {\rm{Aut}}_0$ is isomorphic to ${\rm{D}}_4$.
\end{proof}

\subsection{Einstein-Weyl spaces}
We end this section by reconciling the automorphisms 
found using twistor theory with the automorphisms given 
in Theorem \ref{summarytheorem}, and also proving that 
$\tilde{\Lambda}(\vartheta)$ defined in Section \ref{explicit}
is indeed a conformal map.  To do this, we need to study more 
closely the associated Einstein-Weyl spaces of the $G_1$ and 
$G_3$ actions on the twistor space. 
Recall that in the proof of Lemma \ref{lemma-inv7}, we
defined two linear projections $f_j:\mathbb{CP}^5\to\mathbb{CP}^3$ 
$(j=1,3)$ whose restriction to $\tilde{Z}$ can be viewed as the 
quotient map with respect to the $G_i$-action.
Also recall that the images $f_j(\tilde{Z})$ are non-singular 
quadrics whose equations are given by
\begin{align}
\label{image1'}
f_1(\tilde{Z})=\{\alpha^2 z_2^2+z_3^2+2w_4w_5=0\},\,\,
f_3(\tilde{Z})=\{2w_0w_1+\beta^2z_2^2+z_3^2=0\}.
\end{align}

For fibers of $f_1$ and hyperplane sections of the image $f_1(\tilde{Z})$, 
we have the following 
\begin{lemma}\label{lemma-images}
(i) The closures of general fibers of $f_1$  are smooth conics.
(ii) If $h$ is a $G_3$-invariant plane in $\CP^3$, 
the inverse image $f_1^{-1}(h)$ is reducible if and only if 
$h=\{z_2=\pm \i z_3/\alpha\}$ or $\{z_2=\pm \i z_3/\beta\}$
\end{lemma}

Since everything is explicit, we omit a proof of the lemma.
Of course, an analogous result holds for the other quotient map $f_3$.
We also note that the three involutions on $\CP^5$
determined by the matrices $\Lambda_1, \Lambda_2$ (defined in 
\eqref{iota1and2}), and $\Lambda_3:=\Lambda_2  \Lambda_1$ 
naturally descend to the target space for both of the quotient maps.

By \cite[Section 7]{LeBrun1991}, the minitwistor lines of these 
minitwistor spaces are precisely the hyperplane sections 
$h\cap f_j(\tilde{Z})$, where the plane $h$ satisfies

\vspace{1mm}
\noindent
(A) $h$ is real with respect to the naturally induced real 
structure on $\CP^3$ (so that the real locus on $h$ is 
necessarily $\RP^2$).

\noindent
(B) $h\cap f_j(\tilde{Z})$ does not contain a real point.

\vspace{1mm}\noindent
In other words, the 3-dimensional Einstein-Weyl space  
appears as the parameter space of these planes.
In particular, since the  involutions $\Lambda_1,\Lambda_2$ 
and $\Lambda_3$ naturally induce those on $\CP^3$ 
as above, these also induce involutions on $\H^3$,
which we denote by $\phi_1, \phi_2, \phi_3$, respectively. 
For the purpose of writing these down in explicit form, next we  
determine all the  planes $h$ satisfying (A) and (B):
\begin{lemma}\label{lemma-hball}
(i) Any  plane in  $\CP^3$ having $(z_2,z_3,w_4,w_5)$ 
as  homogeneous 
coordinates as in \eqref{proj1} satisfying the above 
conditions (A) and (B), is of the form
\begin{align}\label{r-planes}
z_2=\i bz_3+cw_4-\ol cw_5,
\end{align}
where $b\in \RR,\,c\in \CC$ satisfy the following inequality:
\begin{align}\label{hball}
b^2+2|c|^2<\frac{1}{\alpha^2}.
\end{align}
(ii) Alternatively any  plane in  $\CP^3$ having $(w_0,w_1,z_2,z_3)$ 
as  homogeneous 
coordinates as in \eqref{proj2} satisfying the 
conditions (A) and (B), is either of the form
\begin{align}\label{r-planes3}
z_2=\i b'z_3+c'w_0+\ol c'w_1,
\end{align}
where $b'\in \RR,\,c'\in \CC$ satisfy the inequality
\begin{align}\label{hball3}
(b')^2-2|c'|^2>\frac{1}{\beta^2},
\end{align}
or otherwise of the form
\begin{align}\label{hball3'}
z_3=cw_0-\ol c w_1,
\end{align}
where $c\in \CC$ satisfies $|c|^2<1/2$.
\end{lemma}
\begin{proof}
Since the real structure on $\CP^3$ is given by 
\begin{align}\label{rson1}
(z_2,z_3,w_4,w_5)\mapsto (\ol z_2,-\ol z_3,-\ol w_5,-\ol w_4),
\end{align}
a plane $h=\{az_2+bz_3+cw_4+dw_5=0\}$ is real if and 
only if $a\in \RR, b\in \i \RR, d=-\ol c$.
It can be verified by simple computations that if $a=0$, 
$h\cap f_1(\tilde{Z})$ always contains real points. Hence we may suppose
\begin{align}\label{bc}
h=\{z_2=\i bz_3+cw_4-\ol cw_5\},\,\,b\in \RR,\,\,c\in \CC.
\end{align}
Substituting into \eqref{image1'},  putting $w_5=-\ol w_4$ and 
replacing $z_3$ by $iz_3$ using the reality requirement, the 
condition (B) is equivalent to the condition that the equation
\begin{align}
\alpha^2(-bz_3+cw_4+\ol cw_4)^2-z_3^2-|w_4|^2=0
\end{align} 
has no solution in $(z_3,w_4)\in \RR \times \CC$.
If we write $c=c_1+\i c_2$ and $w_4=x+\i y$, the left hand side can 
be seen to be equal to 
\begin{align}
(\alpha^2 b^2-1)\left(z_3-\frac{2\alpha^2b}{\alpha^2 b^2-1}(c_1x-c_2y)\right)^2
\notag\\
-2\,\frac{2\alpha^2c_1^2+\alpha^2b^2-1}{\alpha^2 b^2-1}
\left(x-\frac{2\alpha^2c_1c_2}{2\alpha^2c_1^2+\alpha^2 b^2-1}y\right)^2\notag\\
-2\,\frac{2\alpha^2c_1^2+2\alpha^2c_2^2+\alpha^2 b^2-1}
{2\alpha^2c_1^2+\alpha^2 b^2-1}y^2.\label{square1}
\end{align}
The condition is equivalent to the definiteness of \eqref{square1}, 
viewed as a real quadratic form of $(z_3,x,y)$.
If this is positive definite, we have $\alpha^2b^2-1>0$ from the 
first term. 
But then the coefficient of $y^2$ necessarily becomes negative, 
contradicting the definiteness.
Hence \eqref{square1} must be negative definite. Hence we have 
$\alpha^2b^2-1<0$.
Then looking the coefficient of the second square, we obtain 
$ 2\alpha^2c_1^2+\alpha^2b^2-1<0$. Then by negativity of the 
coefficient of $y^2$, we obtain 
$2\alpha^2c_1^2+2\alpha^2c_2^2+\alpha^2 b^2-1<0$. 
Conversely, if this last equality holds, all of the three 
coefficients are easily seen to be negative. Thus the quadratic 
form \eqref{square1} is definite if and only 
$2\alpha^2c_1^2+2\alpha^2c_2^2+\alpha^2 b^2-1<0$. 
This is equivalent to \eqref{hball}, and we obtain (i).

The claim (ii) can be argued in a similar way, as long as we notice 
that the real structure on $\mathbb{CP}^3$ with the coordinates 
$(w_0,w_1,z_2,z_3)$ is given by 
$(w_0,w_1,z_2,z_3)\mapsto (\ol w_1,\ol w_0,\ol z_2,-\ol z_3)$, 
which is slightly different form from \eqref{rson1}.
We omit the details of the computations, as they are
similar to the above. 
\end{proof}
The region defined by \eqref{hball} is an ellipsoid, which 
we will denote by $\mathcal{B}(\alpha)$.
Although the region defined by \eqref{hball3} is disconnected, 
it becomes connected by adding the last disc $\{|c|^2<1/2\}$, 
and we will denote this connected region by $\tilde{\mathcal{B}}(\beta)$. 
Lemma \ref{lemma-hball} says that the planes satisfying (A) and (B) 
are parameterized by the ellipsoid $\mathcal{B}(\alpha)$, 
for $f_1(\tilde Z)$, and by the region $\tilde{\mathcal{B}}(\beta)$ 
for $f_3(\tilde Z)$. If we think of the 
Einstein-Weyl space as the space of real hyperplane sections of the 
minitwistor space, these regions naturally appear for the two 
semi-free $\U(1)$-actions, rather than the 
upper-half space model, as long as we adopt the present coordinates.
By \cite[Theorem 2]{LeBrun1993}, an Einstein-Weyl structure 
is naturally induced on these regions and it is precisely the hyperbolic 
structure. Using this, it is now easy to explicitly write down the three 
involutions $\phi_1,\phi_2$ and $\phi_3$ on the Einstein-Weyl space 
$\mathcal{B}(\alpha)$ (with respect to $G_1$):
\begin{lemma}\label{lemma-fl2}
For $(b,c)\in\mathcal{B}(\alpha)$, we have
\begin{align}
\label{phirep}
\phi_1(b,c)=(-b,-\ol c), \ \ \phi_2(b,c)=(-b,c), \ \ \phi_3(b,c)=(b,-\ol c).
\end{align}
Furthermore, the image of the two isolated fixed points of 
the $K_1$-action on $2 \# \CP^2$ 
(the monopole points) under the quotient map to $\mathcal{B}(\alpha)$ 
are given by $(b,c)=(\pm1/\beta,0)$. The images of the two isolated 
fixed points of the $K_3$-action are given by $(b',c')=(\pm1/\alpha,0)$.
\end{lemma}
\begin{proof}
The formulas for $\phi_j$ immediately follow from 
\eqref{r-planes} and the explicit forms of 
$\Lambda_1,\Lambda_2$, and $\Lambda_3$ on $\CP^5$.
The second statement follows from Lemma \ref{lemma-images} (ii).
\end{proof}
Since the $K_3$-action acts by isometries on $\mathcal{B}(\alpha)$, the 
fixed locus of $K_3$ must be a hyperbolic geodesic in  $\mathcal{B}(\alpha)$. 
By Lemma \ref{lemma-fl2}, this geodesic 
contains the monopole points. The formulas \eqref{phirep} then clearly 
imply that the involutions 
$\phi_j$ induced by $\Lambda_j$ correspond exactly with those in 
Theorem \ref{summarytheorem}.

In conclusion, we show that the maps $\tilde{\Lambda}(\vartheta)$ 
defined in Subsection \ref{extra} above are conformal automorphisms. 
We first define 
\begin{align}
\Lambda(\vartheta) = B(e^{i \vartheta}) \Lambda C(e^{- i \vartheta}  ),
\end{align}
recalling the diagonal matrices $B(s)$ and $C(t)$ defined in the 
proof of Proposition \ref{bcprop}. 
\begin{theorem}
\label{final} For any angle $\vartheta$, $\Lambda(\vartheta)$ is
an involution of the twistor space, which induces a conformal involution 
of $[g_{\LB}]$. The induced involution is $\tilde{\Lambda}(\vartheta + \pi/2)$, 
thus the map $\tilde{\Lambda}(\vartheta + \pi/2)$ is a conformal 
automorphism of $(2 \# \CP^2, [g_{\LB}])$. 
\end{theorem}
\begin{proof}It is easy to see that $\Lambda(\vartheta)$ is also an 
involution. For the moment, let us consider only $\Lambda$. 
We first note that the involution $\Lambda$  induces a diffeomorphism 
from $\H^2$ to itself. 
To see this, we argue as follows: in the $6 \times 6$ matrix representation, 
the involution is off-diagonal type. The middle coordinates 
$(z_2, z_3)$ in Section \ref{twistor} can be regarded as a 
(homogeneous) coordinate on  the quotient space
$Z/(\CC^* \times \CC^*)\simeq\mathbb{CP}^1$, while $\H^2$ is the space of 
maximal orbits in the quotient space $2 \# \CP^2 /K$. 
By the explicit form of the matrix $\Lambda$ and the $\CC^*\times\CC^*$-action (given in \eqref{G-action1})
these involutions map $\CC^* \times \CC^*$-orbits 
to $\CC^* \times \CC^*$-orbits, which means that the involution 
is indeed a lift of some diffeomorphism of $\H^2$.

By \cite[Theorem 9.1]{Fujiki2000}, the induced involution 
on $\mathcal H^2$ must be a hyperbolic isometry. To see this,
we first note that as the coordinate $z$ in the equation (53)        
on \cite[page 276]{Fujiki2000}
is a non-homogenous coordinate on the parameter space of the pencil $|F|^K$ 
(consisting of torus-invariant members of the system $|F|$),
and since the same is true for 
the coordinate $z_3/ z_2$ of ours,
it follows that $z$ in Fujiki's paper is related to $z_3/z_2$ by a 
fractional transformation.
(It is possible to write the precise relation between
these two coordinates; 
but we do not need the explicit form).
On the other hand \cite[Theorem 9.1]{Fujiki2000} states that the coordinate $z$
can be used as a conformal coordinate on $\H^2$.
This means that any conformal automorphism of Poon's metric
on $2 \# \CP^2$ (which is of course a special form of Joyce metrics) 
induces a conformal map on $\H^2$ as long as
the automorphism descends to a map on $\H^2$.
Since the conformal group of $\H^2$ is equal 
to the isometry group, this implies the involution must be a 
hyperbolic isometry.  

We next discuss the angular transformation induced by $\Lambda$. 
The $K$-action on $\CP^5$ in \eqref{G-action1}
naturally induces $K_3\simeq K/K_1$-action on $\CP^3=\{(z_2,z_3,w_4,w_5)\}$,
which is explicitly written as 
\begin{align}\label{K3onH}
(z_2,z_3,w_4,w_5)\longmapsto (z_2,z_3,tw_4,t^{-1}w_5), \,\,\,t\in K_3.
\end{align}
This $K_3$-action naturally
induces the (dual) action on the dual space $(\CP^3)^*$.
If $(a,b,c,d)$ means the dual coordinates as before,
the action is concretely given by
$(a,b,c,d)\longmapsto (a,b,t c,t^{-1}d).$
By putting $a=1$ and using $(b,c,d)$ as non-homogeneous coordinates,
the action can be written as 
\begin{align}\label{K3onH2}
(b,c,d)\longmapsto (b,t c,t^{-1}d).
\end{align}
Then recalling $b\in\mathbb R$ and $d=-\ol c$ on the real locus, we obtain that 
the $K_3$-action on $\mathcal{B}(\alpha)$ 
is given by 
\begin{align}\label{K3onH3}
(b,c)\longmapsto (b,t c).
\end{align}
Then since this must be an isometric $\U(1)$-action on the hyperbolic space, 
and since any non-trivial isometric $\U(1)$-action must be rotations around a 
geodesic,
\eqref{K3onH3} means that ${\rm{Arg}}(c)$ can be used as a coordinate on the 
hyperbolic space $\mathcal B(\alpha)\simeq\mathcal H^3_1$.
Then  ${\rm{Arg}}(t)$ can be naturally identified with the coordinate 
$\theta_3$, where $\theta_3$ 
is the coordinate on $\U(1)\simeq K_3$ we have used throughout 
Section \ref{explicit}.
  
Similarly, replacing the role of $K_1$ and $K_3$ in the above argument,
we first obtain that $K_1\,(\simeq K/K_3)$ naturally acts on 
$\CP^3=\{(w_0,w_1,z_2,z_3)\}$ by
$(w_0,w_1,z_2,z_3)\mapsto (sw_0,s^{-1}w_1,z_2,z_3)$.
Taking the dual, we obtain the $K_1$-action on $(\CP^3)^*$ 
equipped with dual coordinates 
$(c',d',a',b')$  given by
$(c',d',a',b')\mapsto (s c',s^{-1}d',a',b')$. 
On the locus $a'\neq0$  if we use $(b',c',d')$ as non-homogenous coordinates 
by putting $a'=1$, 
the action is written  as $(b',c',d')\mapsto (b',s c',s^{-1} d')$.
Therefore ${\rm{Arg}} (s)$ can be naturally identified with the 
coordinate $\theta_1$, 
where $\theta_1$ is the coordinate on $\U(1)\simeq K_1$ we used in 
Sections \ref{Lansatz} and \ref{explicit}.

The involution $\Lambda:\CP^5\to\CP^5$ induces an isomorphism 
from $\CP^3$ with coordinates  
$(z_2,z_3,w_4,w_5)$ to $\CP^3$ with coordinates $(w_0,w_1,z_2,z_3)$, 
which is given by 
 \begin{align}\label{Lambda_on_CP^3*-s}
(z_2,z_3,w_4,w_5)\longmapsto
 (w_0,w_1,z_2,z_3)=
 (\beta w_4,-\beta w_5, -iz_3,i\alpha\beta z_2).
 \end{align}
This induces an isomorphism between  the dual spaces which is given by
 \begin{align}\label{dual2}
(c',d',a',b')\longmapsto  (a,b,c,d)=
 (i\alpha\beta b',-ia',\beta c',-\beta d').
 \end{align}
 In the above non-homogeneous coordinates on these two 
$(\mathbb{CP}^3)^*$-s, this can be written as
\begin{align}\label{dual3}
(b',c',d')\longmapsto (b,c,d)=
\left(-\frac{1}{\alpha\beta b'},-\frac{ic'}{\alpha b'},\frac{i d'}{\alpha b'}\right).
\end{align}
Restricting to the real locus, we obtain 
\begin{align}\label{dual4}
\mathbb R\times\mathbb C\ni(b',c')\longmapsto (b,c)=\left(-\frac{1}{\alpha\beta b'},
-\frac{ic'}{\alpha b'}\right)\in\mathbb R\times\CC.
\end{align}
In particular, $\Lambda^*c'=-i c'/(\alpha b')$.
Therefore under $\Lambda^*$,  the two angular coordinates $\theta_1$ and $\theta_3$ are 
related by  $\theta_1=\theta_3+(3\pi/2)$.
Equivalently, this says that the angular action induced by $\Lambda$ is given 
by
\begin{align}
(\theta_1, \theta_3) \mapsto (\theta_3 - \pi/2, \theta_1 + \pi/2). 
\end{align}
Since the angular map induced by $\Lambda$ is orientation-reversing, 
the induced hyperbolic isometry must also be orientation-reversing.  
Since the map $L(\zeta)$ defined above in \eqref{mobius} is the 
unique orientation-reversing isometry with the correct properties 
(see Remark \ref{Lunique}), $\Lambda$ must induce the map 
$\tilde{\Lambda}(\pi/2)$. This clearly implies that 
$\Lambda({\vartheta})$ induces the map $\tilde{\Lambda}(\vartheta + \pi/2)$,
and the proof is complete. 
\end{proof}

\end{document}